\documentclass[11pt, reqno]{amsart}
\usepackage[margin=1in]{geometry}
\usepackage{amsmath}
\usepackage{amsthm}
\usepackage{amsfonts}
\usepackage{amssymb}

\usepackage{appendix}
\usepackage{xcolor}

\usepackage{subcaption}
\usepackage{pdfsync}

\usepackage{stackrel}

\usepackage{microtype}

\usepackage{hyperref}

\usepackage{graphicx}

\numberwithin{equation}{section}

\def\<{\langle}
\def\>{\rangle}

\def\ker{\mathrm{Ker}\,}

\def\ve{\varepsilon}

\DeclareMathOperator*{\esssup}{ess\,sup\,}

\def\dx{ \ dx }
\def\d{ \, d }

\def\R{\mathbb{R}}

\theoremstyle{plain}
\newtheorem{theorem}{Theorem}[section]

\newtheorem{proposition}[theorem]{Proposition}

\newtheorem{lemma}[theorem]{Lemma}
\newtheorem{corollary}[theorem]{Corollary}

\theoremstyle{definition}

\newtheorem{remark}[theorem]{Remark}

\title[Standing waves for Schr\"odinger equations...]{Standing waves for Schr\"odinger equations with Kato class potentials and $L^\infty$-bounded nonlinearities}

\author{Aleksander \'{C}wiszewski}
\address{\newline
Aleksander \'{C}wiszewski
\newline
ORCID-ID: 0000-0003-0662-2065
\newline
Faculty of Mathematics and Computer Science \newline Nicolaus Copernicus University \newline 87-100 Toru\'n, Poland}
\email{aleks@mat.umk.pl}

\author{Piotr Kokocki}
\address{\newline 
Piotr Kokocki
\newline
ORCID-ID: 0000-0002-7763-5624
\newline
Faculty of Mathematics and Computer Science \newline Nicolaus Copernicus University \newline 87-100 Toru\'n, Poland}
\email{pkokocki@mat.umk.pl}

\subjclass[2000]{47J35, 47J15, 37L05} 
\keywords{Semigroup, evolution equation, Conley index, bounded solution, connecting orbit, standing wave, Schr\"{o}dinger equation}

\begin{document}

	
\begin{abstract}
We establish the existence of standing waves for a nonlinear Schr\"{o}dinger
equation with potentials belonging to the Kato class and an $L^\infty$-bounded
nonlinearity, whose Lipschitz constant is smaller than the distance from zero
to the essential spectrum of the linear part. We consider both the nonresonant and resonant cases.
Our approach is based on the Conley index theory applied to study invariant sets of the associated parabolic semiflow. Using properties of the Schr\"{o}dinger semigroup with Kato-class potentials, which follow from its Feynman--Kac representation, we derive
a~priori estimates for bounded solutions in the $L^\infty$ and $L^2$ norms,
as well as regularity bounds in Sobolev spaces. As a consequence, we obtain conditions ensuring the existence of connecting orbits between stationary solutions, which in turn yield the
existence of nontrivial standing waves.
\end{abstract}
	
\maketitle
	
\section{Introduction}
In this paper, we investigate the existence of solutions to the elliptic equation
$$
-\Delta u + V(x)u - \lambda u = f(x,u), \quad x\in\R^N, \leqno{(E)_\lambda}
$$

\noindent where $\lambda$ is a real parameter, $V:\R^N\to\R$ is a potential, and $f:\R^N\times\R\to\R$ is a nonlinearity. Equation $(E)_\lambda$ arises as the stationary problem associated with the nonlinear Schr\"odinger equation
\begin{equation}\label{NLS}
	i\psi_t = -\Delta \psi + V(x)\psi + g(x,|\psi|^2)\psi, \quad t>0,
\end{equation}
which appears in various physical contexts, such as nonlinear optics and Bose--Einstein condensation.
Indeed, standing wave solutions of \eqref{NLS}, sought in the form
\[
\psi(t,x)=e^{-i\lambda t}u(x), \quad t>0, \ x\in\R^{N},
\]
reduce \eqref{NLS} to the elliptic equation $(E)_\lambda$ for the real-valued profile function $u$. The analysis of such equations becomes significantly more involved when the potential $V$ exhibits singularities or lacks regularity. In this work, we focus on potentials that satisfy the condition 
\begin{align}\label{def-k-3}
		\lim_{\ve\to 0^{+}} \left(\sup_{x\in\R^{N}} \int_{\{|x-y|\le\ve\}}\frac{|V(y)|}{|x-y|^{N-2}}\,dy\right) = 0 & \quad \text{if} \ \ N\ge 3;\\[3pt]
		\label{def-k-2}
		\lim_{\ve\to 0^{+}} \left(\sup_{x\in\R^{N}} \int_{\{|x-y|\le\ve\}} \ln|x-y|^{-1}|V(y)|\,dy\right) = 0 & \quad \text{if} \ \ N=2; \\[3pt] 
		\label{def-k-1}
		\sup_{x\in\R^{N}} \int_{\{|x-y|\le 1\}}|V(y)|\,dy <\infty & \quad \text{if} \ \ N=1.
	\end{align}
These conditions characterize the \emph{Kato class} of potentials $K_{N}$, which provides a framework for treating Schr\"odinger operators with physically relevant singularities. A well-known subclass of $K_N$ consists of the Rellich–Kato potentials, given by
$$\leqno{(K\!R)}\qquad \qquad \left\{\begin{gathered} V=V_\infty+V_0, \ \text{where} \ V_\infty \in L^\infty(\R^N) \ \text{and} \ V_0\in L^p(\R^N) \ \text{with $p$ satisfying}\\ p\geq 2 \ \text{ if } \ 1\leq N\leq 3 \ \ \text{ and } \ \ p>N/2 \ \text{ if } \ N\geq 4. \end{gathered}\right.$$
The class $(K\!R)$ includes, in particular, Coulomb-type potentials $V(x) = a|x|^{-1}$, for $a \in \mathbb{R}$, which model the electron--nucleus interaction in the hydrogen atom. A natural generalization is provided by the potential
\begin{align}\label{v-atomic}
V(x) = - \sum_{i=1}^{n} \frac{Z}{|x_i|} + \sum_{1\le i < j \le {n}} \frac{1}{|x_i - x_j|}, \quad x = (x_1, \ldots, x_{n}) \in (\mathbb{R}^{3})^n, 
\end{align}
describing $n$ electrons ($n\ge 2$) interacting with a nucleus charge $Z > 0$ fixed at the origin, together with mutual Coulomb repulsion between electrons. Such potentials arise naturally in the mathematical description of many-particle quantum systems and constitute one of the principal examples of Kato-class interactions that are not covered by the Rellich--Kato assumptions (see \cite{Aiz-Sim}, \cite{MR0670130}). For the atomic potentials \eqref{v-atomic}, it was shown in \cite{zhi} that there exists $\nu \leq 0$ such that $\sigma_{\mathrm{ess}}(-\Delta + V) = [\nu, +\infty)$. Moreover, if $Z>n-1$ then the discrete spectrum is non-empty and consists of an infinite sequence of eigenvalues $(\lambda_k)_{k \geq 1}$ accumulating at $\nu$ from below. Further details on spectral properties of atomic Schrödinger operators can be found in \cite{evans} and \cite{MR1768629}.

In our study of equation $(E)_\lambda$, we further assume that the function $f$ satisfies the conditions.\\[2pt]
\noindent\makebox[9mm][r]{$(f1)$} \parbox[t][][t]{155mm}{For every $u \in \R$, the mapping $x \mapsto f(x,u)$ is Lebesgue measurable, and the function $c := f(\,\cdot\,,0)$ belongs to $L^{2}(\R^{N})$.}\\[2pt]
\noindent\makebox[9mm][r]{$(f2)$} \parbox[t][][t]{155mm}{There exists  $l\in L^\infty(\R^N)$ such that 
\begin{equation*}
|f(x,u)-f(x,v)|\leq l(x)|u-v|\quad \text{for a.e. } x\in\R^N \text{ and all } u,v\in\R.
\end{equation*}}\\[2pt]
\noindent\makebox[9mm][r]{$(f3)$} \parbox[t][][t]{155mm}{The Lipschitz bound $l$ satisfies
\begin{equation*}
\hat \varrho (l) < \inf \sigma_{\mathrm{ess}}(-\Delta + V) - \lambda, 
\end{equation*}
where 
\begin{equation*}
\hat \varrho (l) := \lim_{R\to \infty} \esssup_{|x|\ge R} l (x).
\end{equation*}}\\	
To study the existence of solutions to $(E)_\lambda$, we consider the following time-dependent equation 
$$
u_t = \Delta u - V(x)u +\lambda u + f(x,u),\quad  x\in\R^N,\ t>0. \leqno{(P)_\lambda}
$$ 

\noindent The Schr\"{o}dinger operator  $-\Delta+V$ is self-adjoint in $L^2(\R^N)$ and bounded from below. 
Consequently, if $\inf\sigma_{\mathrm{ess}}(-\Delta + V)>\lambda$, then the spectrum of $-\Delta+V$ below $\lambda$ consists of at most finitely many eigenvalues.
This allows us to define the total multiplicity of eigenvalues less than $\lambda$ by
\begin{equation*}
d^{-}(V,\lambda) := \sum_{\mu \in \sigma(-\Delta + V) \cap (-\infty, \lambda)} \dim \ker(-\Delta + V - \mu).
\end{equation*}
The main result in the nonresonant case is the following criterion.
\begin{theorem}\label{14102025-1145}
Let $f$ satisfy the conditions $(f1)$\,--\,$(f3)$ with $f(x,0)=0$ for a.e.\
$x\in\R^N$, and let $\alpha, \omega\in L^{\infty}(\R^{N})$ be such that
$$\lim_{u\to 0} \frac{f(x,u)}{u} = \alpha(x) \quad \text{for a.e. }
x\in\R^N, \leqno{(L_0)}$$
and
$$\lim_{|u|\to \infty} \frac{f(x,u)}{u} = \omega(x) \quad \text{for a.e. }
x\in\R^N. \leqno{(L_\infty)}$$
Assume that $\lambda$ belongs to the resolvent set of both $-\Delta+V-\alpha$
and $-\Delta+V-\omega$, and that
\begin{equation}\label{cond-alpha}
\max\{\hat\varrho(|\alpha|), \hat\varrho(|\omega|)\} < \inf \sigma_{\mathrm{ess}}
(-\Delta+V)-\lambda \quad\text{and}\quad d^{-}(V-\alpha, \lambda)\neq
d^{-}(V-\omega, \lambda).
\end{equation}
Then there exist a nonzero solution $\bar u \in H^2(\R^N)$ of $(E)_\lambda$
and a bounded solution $u : \R \to H^1(\R^N)$ of $(P)_\lambda$ such that
either $\bar u$ belongs to the $\alpha$-limit set of $u$ and $u(t) \to 0$ in
$H^1(\R^N)$ as $t \to +\infty$, or $\bar u$ belongs to the $\omega$-limit set
of $u$ and $u(t) \to 0$ in $H^1(\R^N)$ as $t \to -\infty$.
\end{theorem}
Now let us turn to the resonant case, where we assume that $\lambda$ is an
eigenvalue of $-\Delta+V$ and that there exists $m \in L^{\infty}(\R^{N})$
such that
\begin{equation}\label{bound-f}
|f(x,u)| \leq m(x) \quad \text{for a.e. } x \in \mathbb{R}^N \text{ and for all } u \in \mathbb{R}.
\end{equation}
In this setting, we impose additional conditions on the behavior of $f$ at
infinity, namely the \emph{Landesman--Lazer} type conditions, which require
that either
$$
\int_{\{\varphi >0\}} \check f_+(x) \varphi (x) \,d x  + \int_{\{\varphi <0\}} \hat f_-(x) \varphi (x) \,d x > 0 \
\ \mbox{ for }\  \varphi \in\ker(-\Delta+V-\lambda)\setminus\{0\}, \leqno{(LL)_+}
$$
or
$$
\int_{\{\varphi>0\}} \hat f_+(x) \varphi (x) \,d x + \int_{\{\varphi <0\}} \check f_-(x) \varphi (x) \,d x < 0
\ \ \mbox{ for }\  \varphi \in\ker(-\Delta+V-\lambda)\setminus\{0\}, \leqno{(LL)_-}
$$

\noindent where $\hat{f}_\pm(x)= \limsup_{s \to \pm\infty} f(x,s)$ and  $\check{f}_\pm (x)= \liminf_{s \to \pm\infty} f(x,s)$ for $x\in\R^N$. 
In the special case where $V$ is of Rellich--Kato type, the unique
continuation principle holds for the Schr\"odinger operator (see
\cite[Th.\,1.1]{MR882069} and \cite[Prop.\,3]{Gossez}). As a consequence,
conditions $(LL)_+$ and $(LL)_-$ are implied by the following weaker
conditions
$$
\left\{
\begin{array}{c}
	\check{f}_+(x) \geq 0 \ \mbox{ and } \ \hat{f}_-(x)\leq 0 \ \mbox{ for a.e. } \ x\in\R^N, \\
		\text{there exists a set of positive measure on which both} \ \check{f}_+ > 0 \ \text{and} \ \hat{f}_- < 0,
	\end{array}
	\right.
\leqno{\widetilde{(LL)}_+}
$$
or
$$
\left\{
\begin{array}{c}
\hat{f}_+(x)\leq 0 \ \mbox{ and } \ \check{f}_- (x)\geq 0 \ \mbox{ for a.e. } \ x\in\R^N,\\
\text{there exists a set of positive measure on which both} \ \check{f}_+ < 0 \ \text{and} \ \hat{f}_- > 0.
\end{array} \right. \leqno{\widetilde{(LL)}_-}
$$

\noindent The main result in the resonant case is the following theorem.
\begin{theorem}\label{30042019-1204}
Let $f$ satisfy the conditions $(f1)$\,--\,$(f3)$ and \eqref{bound-f} with $f(x,0)=0$ for a.e.\ $x\in\R^N$, and suppose there exists $\alpha \in L^\infty(\R^N)$ satisfying $(L_0)$. Assume that $\lambda$ is an eigenvalue
of $-\Delta+V$ such that
$$\hat \varrho (|\alpha|) < \inf \sigma_{ess} (-\Delta+V) - \lambda$$ and that one of the following conditions is satisfied: \\[3pt]
\noindent\makebox[7mm][r]{$(i)$} \parbox[t][][t]{155mm}{$(LL)_+$ holds and $d^-(V,\lambda) +\dim \ker (-\Delta+V-\lambda) \neq d^{-} (V-\alpha,\lambda)$,}\\[3pt]
\noindent\makebox[7mm][r]{$(ii)$} \parbox[t][][t]{155mm}{$(LL)_-$ holds and $d^-(V,\lambda) \neq d^{-} (V-\alpha,\lambda)$.}\\[3pt]
Then there exist a nonzero solution $\bar u\in H^2(\R^N)$ of $(E)_\lambda$ and a bounded solution $u:\R\to H^1(\R^N)$ of $(P)_\lambda$ such that either $\bar u$ belongs to the $\alpha$-limit set of $u$ and $u(t)\to 0$ in $H^1(\R^N)$ as $t\to -\infty$, or $\bar u$ belongs to the $\omega$-limit set of $u$ and $u(t)\to 0$ in $H^1(\R^N)$ as $t\to +\infty$. 
\end{theorem}
Both criteria stated in Theorems~\ref{14102025-1145} and~\ref{30042019-1204} generalize the corresponding results of \cite{Cw-Kok}, where $V$ was
assumed to be the Kato--Rellich class potential. The present framework allows for potentials from the broader Kato class and, in particular, includes many-particle Coulomb interactions \eqref{v-atomic} and it also opens for a significant broad class of nonlinearities. Moreover, here we do not impose any geometric conditions on $V$; instead, they are replaced by condition $(f3)$, which involves the asymptotic relationship of the Lipschitz constant $l$ and the essential spectrum of $-\Delta+V$. Observe that in the non-resonant case in Theorem \ref{14102025-1145}, we assume that $f$ possesses an asymptotic potential $\omega$, whereas in the resonant case the nonlinearity is bounded by a function $m$, with both $\omega$ and $m$ belonging to
$L^\infty(\R^{N})$. The assumption that $m$ is merely bounded is one of the main sources of difficulties in our analysis but extends the scope of our considerations, since these assumptions are particularly relevant for the
nonlinear Schr\"odinger equation~\eqref{NLS}, in which the dynamics are
governed not only by the external potential $V$ but also by 
\emph{saturable nonlinearities}, which satisfy
\[
g(x,u) \to g_\infty^{\pm}(x) \quad \text{as } u \to \pm\infty,
\]
uniformly in $x$ on bounded subsets of $\R^N$; see, e.g., \cite{Karja,
Sulem}. Such nonlinearities arise naturally in various physical models, where the response of the medium remains bounded for large amplitudes of the wave function. For instance, the exponential nonlinearity $g(u) = 1 - e^{-u}$ appears in the context of turbulent plasma wave dispersion (see \cite{LaedkeSpatschek1984}, \cite{lash-cher}), while the nonlinearity $g(u) = -au(1+u)^{-1}$, with $a\neq 0$, models light packets in semiconductor-doped glass (see \cite{gatz-herman}, \cite{liao-zhang}). Nonlinearities of the form $g(u)=b(1+u)^{-1}$, with $b \neq 0$, appear in the analysis of wave propagation in photorefractive materials (see \cite{Gabriel2009}, \cite{MR2451606}). We also refer to \cite{MR3427792} for variational results on the existence of standing waves for the Schr\"odinger equation with general saturable nonlinearities. In all of these examples, the nonlinear term naturally fits into the $L^\infty$ framework considered in this paper. 


Our methods are based on the Conley index, in the version developed by Rybakowski (see \cite{rybakowski, rybakowski-TAMS}), applied to the parabolic semiflow $(P)_\lambda$. 
In contrast to previous works (see, e.g., \cite{Cw-Kok}, \cite{Cw-Kr-2019}, \cite{MR3072663}, \cite{MR4293056}, \cite{MR1992823}, \cite{MR2387828}), where the assumption $m \in L^2(\mathbb{R}^N)$ was imposed, the present setting with $m \in L^\infty(\mathbb{R}^N)$ is substantially more demanding from an analytical point of view. Indeed, on the unbounded domain $\mathbb{R}^N$, the condition \eqref{bound-f} does not provide any $L^2$ control over $H^1$-bounded full solutions of $(P)_\lambda$. Consequently, estimates in the standard energy space $H^1(\mathbb{R}^N)$ cannot be directly exploited when constructing an admissible isolating neighborhood for the maximal bounded invariant set in the proof of Theorem \ref{30042019-1204}. To overcome these difficulties, we develop an alternative approach and establish \emph{a priori} bounds for $H^1$-bounded solutions in the $L^\infty$ norm by combining estimates derived from the Feynman--Kac representation of the Schr\"odinger semigroup generated by $-\Delta + V$ with a duality argument. This allows us to compensate for the lack of global $L^2$ control in the construction of an admissible isolating neighborhood in $H^{1}(\R^{N})$.

The paper is organized as follows. In Section~2, we discuss spectral properties of the Schr\"odinger operator in $L^2(\R^N)$ and recall some estimates for its associated semigroup. Section~3 contains $L^1$ estimates for the semigroup generated by the Schr\"odinger operator. In Section~4, we
derive $L^\infty$ estimates for solutions of $(P)_\lambda$ that are bounded in
$H^1(\R^N)$, while Section~5 is devoted to the corresponding $L^2$ and
$H^1$ estimates. Section~6 addresses the compactness properties of semiflows
generated by perturbations of the Schr\"odinger operator by continuous
nonlinearities satisfying a uniform condensing condition. In Section~7, we
provide proofs of the main results. The Appendix contains some basic information about sectorial operators in Banach spaces and the Conley index in the sense of Rybakowski for infinite-dimensional semiflows.

\section{Properties of the Schr\"odinger operator} \label{sec-2}
Let $A_0$ be the operator on $X := L^2(\mathbb{R}^N)$ defined by
\begin{align*}
D(A_0) := H^2(\mathbb{R}^N) \quad \text{and} \quad A_0 u := -\Delta u \quad \text{for } u \in D(A_0),
\end{align*}
where $\Delta = \sum_{k=1}^{N} \frac{\partial^2}{\partial x_k^2}$ denotes the Laplacian, with derivatives understood in the weak sense. It is well known that $A_{0}$ is a self-adjoint sectorial operator with spectrum $\sigma(A_0) = [0,+\infty)$. Moreover, if $V\in K_N$, then
		\begin{equation}\label{crucial-property-of-Kato}
			\||V|^{1/2} (A_0+\lambda )^{-1/2}\|_{\mathcal{L}(L^2, L^2)}\to 0 \ \ \mbox { as  } \ \ \lambda\to+\infty,
		\end{equation}
where the fractional resolvent is defined by $$(A_{0}+\lambda)^{-1/2} := \frac{1}{\Gamma(1/2)} \int_0^\infty t^{-1/2} e^{-\lambda t}S_{A_{0}}(t) \, dt$$
where $\{ S_{A_0}(t):X\to X\}_{t\geq 0}$ is the $C_0$-semigroup generated by the operator $-A_0$ 
(see \cite[Prop. A.2.3]{MR0670130} and the remarks preceding \cite[Th. A.2.7]{MR0670130}). 
        In particular, \eqref{crucial-property-of-Kato} implies that for every $\ve>0$ there exists a constant $C_\varepsilon>0$ such that
		\begin{equation}\label{A0-bouneded}
			\int_{\R^N} |V(x)| |u(x)|^2 \dx  \le \varepsilon \| \nabla u\|_{L^2}^2  + C_\varepsilon \|u\|_{L^2}^{2}, \quad  u\in H^1(\R^N).
		\end{equation}
		This inequality allows us to define the symmetric bilinear form 
		$$
		q_V(u,v) := \int_{\R^N} \nabla u(x)\nabla v(x) \, dx  + \int_{\R^N} V(x) u(x)v(x) \, dx,  \quad u,v\in H^1(\R^N). 
		$$
		The form is semibounded from below: choosing $\varepsilon=1/2$ in the inequality \eqref{A0-bouneded}, we obtain a constant $C_V>0$ such that
		\begin{equation}\label{bounded-from-below}
			q_V(u,u) \geq \frac{1}{2} \|\nabla u\|_{L^2}^2 - C_V \|u\|_{L^2}^2 \geq - C_V \|u\|_{L^2}^2, \quad u\in H^1(\R^N).
		\end{equation}
		Moreover, we also get the following estimate
		\begin{equation}\label{bounded-from-above}
			q_V(u,u) \leq \frac{3}{2} \|\nabla u\|_{L^2}^2 + C_V\|u\|_{L^2}^{2}, \quad u\in H^1(\R^N).
		\end{equation}
		The form $q_V$ is closed: the space $H^1(\R^N)$, equipped with the norm 
		$$
		\|u\|_{V}^2 := q_V(u,u) + (C_V+1)\|u\|_{L^2}^2, \quad u\in H^1(\R^N)
		$$
		is complete, since by \eqref{bounded-from-below} and \eqref{bounded-from-above} the norm $\|\cdot\|_V$ is equivalent to the standard norm of $H^1(\R^N)$. Hence, by \cite[Th. VIII.15]{MR0493419}, the form $q_V$ determines a unique self-adjoint operator $A:D(A)\subset L^2(\R^N)\to L^2(\R^N)$ such that $D(A)\subset H^{1}(\R^N)$ is a dense subspace and
		$$
		q_V(u,v) = (Au, v)_{L^2}, \quad  u\in D(A), \ v\in H^{1}(\R^{N}).
		$$
		Since $A=-\Delta+V$ is self-adjoint and bounded from below by \eqref{bounded-from-below}, it is sectorial. 
        Setting $\delta := 1 + C_{V}$ ensures that $A + \delta$ is strictly positive, which allows us to define the fractional power $(A + \delta)^{1/2}$. We then consider the fractional power space $X^{1/2} := D((A + \delta)^{1/2})$ equipped with the graph norm $$\|u\|_{1/2}:=\|(A + \delta)^{1/2}u\|_{L^{2}}, \quad u\in X^{1/2},$$
		where 
		$$(A + \delta)^{-1/2} := \frac{1}{\Gamma(1/2)} \int_0^\infty t^{-1/2} e^{-(1 + C_{V})t}S_{A}(t) \, dt.$$
		Then $X^{1/2}$ is a Banach space with $D(A)$ as a dense subspace. Since $(A+\delta)^{1/2}$ is symmetric,
		\begin{equation}\label{eq-1122}
			\|(A + \delta)^{1/2}u\|^{2}_{L^{2}} =  q_V(u,u) + (1+C_{V})\|u\|_{L^2}^{2}, \quad u\in D(A),
		\end{equation}
		which, together with \eqref{bounded-from-below}, \eqref{bounded-from-above}, and the density of $D(A)$ in both $H^1(\R^N)$ and $X^{1/2}$, implies $X^{1/2} = H^1(\R^N)$ and the equivalence of the corresponding norms.
		
		Since $A$ is sectorial, $-A$ generates an analytic $C_0$-semigroup $\{S_A (t)\}_{t\geq 0}$ of bounded linear operators on $X=L^2(\R^N)$. This semigroup admits a probabilistic representation via the Feynman--Kac formula
		\begin{align*}
			[S_A(t) u](x) = E_{x}\!\left[\exp\left(-\int_{0}^{t}V(B_s)\,ds\right)u(B_t)\right]\!,\quad u\in C^{\infty}_{0}(\R^{N}),
		\end{align*}
		where $(B_t)_{t\geq 0}$ is the standard Brownian motion starting at the origin; see \cite{MR0493420}, \cite{MR0670130}, \cite{MR1717054}. The condition $V \in K_N$ guarantees that the exponential factor is integrable and exhibits an appropriate behavior, as $t \to 0^{+}$. This representation serves as a powerful tool for deriving properties of the semigroup, including the following estimates
		\begin{align}\label{est-1}
			& \|S_{A}(t)u\|_{L^{2}} \le  e^{\omega t}\|u\|_{L^{2}},  && \hspace{-70pt}\ u\in L^{2}(\R^{N}), \ t\geq 0,\\ \label{est-2}
			&   \|S_{A}(t)u\|_{L^{\infty}}  \le C e^{\omega t}\|u\|_{L^{\infty}},  &&\hspace{-70pt} \ u\in L^{\infty}(\R^{N}),\ t\geq 0,\\ \label{est-3}
			&  \|S_{A}(t)u\|_{L^{\infty}}  \le C t^{-N/4}e^{ \omega t}\|u\|_{L^{2}}, &&  \hspace{-70pt} \ u\in L^{2}(\R^{N}), \ t>0,
\end{align}
where $C>0$ and $\omega\in\R$ are constants -- see \cite[Th. B.1.1]{MR0670130} and, additionally, due to \cite[Cor. B.3.2]{MR0670130}, if $1\leq p\leq +\infty$, then
\begin{equation}\label{est-1-a}
S_A(t)u\in C(\R^N) \ \mbox{ for any } u\in L^p(\R^N),\ t>0.             
\end{equation}
In the rest of this section, we assume that $\gamma>0$ is an arbitrary number such that
\begin{equation} \label{del-alph-zer}
\gamma < \inf\sigma_{ess}(A) - \lambda \quad \mbox{ and } \quad  \gamma + \lambda \not\in\sigma(A).
\end{equation}
By the spectral theorem for self-adjoint operators, we have the direct sum decomposition 
\begin{equation}\label{eq-decomp}
X = X_- \oplus X_0 \oplus X_{+}^{d} \oplus X_{+}^{e}
\end{equation}
into closed, mutually orthogonal subspaces such that $X_0 = \ker (A-\lambda)$ and
\begin{gather*}
\sigma (A_{X_-}) = \sigma(A)\cap (-\infty,\lambda), \quad \sigma (A_{X_{+}^{d}}) = \sigma(A)\cap (\lambda,\lambda + \gamma), \quad \sigma (A_{X_{+}^{e}})=\sigma(A)\cap(\lambda+\gamma,+\infty),
\end{gather*}
		where $A_{X_-}$, $A_{X_{+}^{d}}$, and $A_{X_{+}^{e}}$ are parts of $A$ in $X_{-}$, $X_{+}^{d}$, and $X_{+}^{e}$, respectively. Observe that $X_-$, $X_0$, and $X_+^d$ are finite-dimensional. Let $Q_{-}$, $P$, $Q_{+}^{d}$, and $Q_{+}^{e}$ be     
        orthogonal projections onto the components of the decomposition \eqref{eq-decomp}. We denote $X_{+}:=X_{+}^{d}\oplus X_{+}^{e}$ and $Q_+:=Q_{+}^{d}+Q_{+}^{e}$. In the following, we verify that the operators $Q_{\pm}$ and $P$ are bounded with respect to the Lebesgue norms.
		\begin{proposition}\label{prop-inclusion}
			For any $1\le p\le \infty$, the following inequalities hold \footnote{ \ We will write $A(u)\lesssim B(u)$, where $A, B$ are non-negative functions dependent on the argument $u$, if there is a constant $C\ge 0$ such that $A(u)\le CB(u)$ for all $u$.}
			\begin{align}\label{est-1aa}
				& \|Q_{-}u\|_{L^{p}}\lesssim \|u\|_{L^{p}},\quad \|Pu\|_{L^{p}}\lesssim \|u\|_{L^{p}},\\
				& \|Q_{+}^{d}u\|_{L^{p}} \lesssim \|u\|_{L^{p}}, \quad \|Q_{+}^{e}u\|_{L^{p}} \lesssim \|u\|_{L^{p}}, 
			\end{align}
			for all $u\in L^{2}(\R^{N})\cap L^{p}(\R^{N})$. 
		\end{proposition}
		\begin{proof}
			Let the vectors $\{e_{k} \ | \ 1 \le k \le N_{-}\}$, with $N_- =\dim X_-$, form an orthonormal basis of the finite-dimensional subspace $X_{-}$. Then
			\begin{align}\label{def-q-p}
				Q_{-}u = \sum_{k=1}^{N_{-}} \alpha_{k}(u)e_{k},\quad u\in L^{2}(\R^{N}),
			\end{align}
			where the functionals $\alpha_{k}$ are defined by
			\begin{align*}
				\alpha_{k}(u) := \int_{\R^{N}}e_{k}(x)u(x)\,dx,\quad 1\le k\le N_{-}.
			\end{align*}
			Since any eigenfunction $v$ corresponding to an isolated eigenvalue of finite multiplicity satisfies 
			\begin{align}\label{est-eigen}
			|v(x)|\le Ce^{-\omega |x|},\qquad x\in\R^{N},
			\end{align} 
			for some constants $\omega>0$ and $C>0$ (possibly depending on $v$) -- see  \cite[Th. C.3.4]{MR0670130} -- we infer that  $e_{k}\in L^{q}(\R^{N})$ for all $q\in[1,\infty]$ and each $k=1,\ldots, N_-$. Hence, if $u\in L^{2}(\R^{N})\cap L^{p}(\R^{N})$, then 
$$|\alpha_{k}(u)|\le \|u\|_{L^{p}}\|e_{k}\|_{L^{p'}}, \quad 1\le k\le N_{-},$$ where $p'$ denotes the conjugate Lebesgue exponent of $p\in[1,\infty]$. Combining this with \eqref{def-q-p}, gives $Q_{-}u\in X_{-}\cap L^{p}(\R^{N})$ and the estimate
			\begin{gather}\label{ineq-q-m}
				\|Q_{-}u\|_{L^{p}} \le \sum_{k=1}^{N_{-}}|\alpha_{k} (u)| \|e_{k}\|_{L^{p}} \le  
				\|u\|_{L^{p}}\sum_{k=1}^{N_{-}}\|e_{k}\|_{L^{p'}} \|e_{k}\|_{L^{p}} \lesssim \|u\|_{L^{p}}.
			\end{gather}
			\indent Since $X_{0}$ and $X_{+}^{d}$ are also finite-dimensional subspaces spanned by eigenfunctions corresponding to isolated eigenvalues, we similarly deduce that $Pu\in X_{0}\cap L^{p}(\R^{N})$, $Q_{+}^{d}u\in X_{+}^{d}\cap L^{p}(\R^{N})$ and that the following estimates hold:
			\begin{gather}\label{ineq-q-m22}
				\|Pu\|_{L^{p}}\lesssim \|u\|_{L^{p}},  \quad  \|Q_{+}^{d}u\|_{L^{p}} \lesssim \|u\|_{L^{p}},\quad u\in L^{2}(\R^{N})\cap L^{p}(\R^{N}).
			\end{gather}
			\indent Finally, combining  $Q_{+}^{e} = I - Q_{+}^{d} - P - Q_{-}$ with \eqref{ineq-q-m} and \eqref{ineq-q-m22} yields
			\begin{align*}
				\|Q_{+}^{e}u\|_{L^{p}} & \le \|u\|_{L^{p}} + \|Q_{+}^{d}u\|_{L^{p}} +\|Pu\|_{L^{p}} + \|Q_{-}u\|_{L^{p}} \lesssim \|u\|_{L^{p}},\quad u\in L^{2}(\R^{N})\cap L^{p}(\R^{N}),
			\end{align*}
			thereby completing the proof of the proposition.
		\end{proof}
		Since the operator $A_{X_+}$ is self-adjoint and $\sigma(A_{X_+}) \subset (\lambda, +\infty)$, it is sectorial (see, e.g., \cite[Prop. 1.3.3]{dlotko-cholewa}). Consequently, $-A_{X_+}$ generates the analytic $C_0$-semigroup that coincides with $S_A(t)$ restricted to $X_+$. Applying \cite[Th. 6.13]{Pazy} and using the equality $X^{1/2} = H^1(\R^N)$ together with the equivalence of the corresponding norms, we obtain constants $\delta_+>0$ and $C_+>0$ such that 
		\begin{align}\label{est-11aa}
		\|S_{A}(t)u\|_{L^{2}} & \le e^{-(\lambda+\delta_{+})t}\|u\|_{L^{2}},  && \hspace{-70pt} t\ge 0, \ u\in X_{+}, \\ \label{est-11bb}
		\|S_{A}(t)u\|_{H^{1}} & \le C_{+} t^{-1/2} e^{-(\lambda+\delta_{+})t}\|u\|_{L^{2}}, && \hspace{-70pt} t > 0, \ u\in X_{+}.
		\end{align} 
Similarly, since $A_{X_{+}^{e}}$ is self-adjoint and $\sigma(A_{X_{+}^{e}})\subset (\lambda+\gamma,+\infty)$, we obtain 
		\begin{align}\label{ine11}
			\|S_{A}(t)u\|_{L^{2}} \le e^{-(\lambda+\gamma) t}\|u\|_{L^{2}}, \quad t\ge 0, \ u\in X_{+}^{e}.
		\end{align}
For the subspace $X_{-}$, the semigroup $\{S_{A_{X_-}} (t)\}_{t\geq 0}$ can be extended to a $C_0$-group of bounded operators, and there exist $\delta_- > 0$ and $C_->0$ such that
\begin{equation}\label{ine33}
\|S_{A_{X_-}}(t)u\|_{L^{2}} \le C_{-} e^{(-\lambda+\delta_{-}) t}\|u\|_{L^{2}},  \quad  t\le 0, \ u\in X_{-}.
\end{equation} 
Using the above estimates, we can derive the following compactness lemma.
\begin{lemma}\label{lem-con-sem}
For any bounded $W\subset L^2(\R^N)$, one has
\begin{equation}\label{rownorm}
\beta_{L^{2}}(S_{A-\lambda}(t)W) \le e^{-\gamma t}\beta_{L^{2}}(W), \quad t\ge 0,
\end{equation}
where $\beta_{L^{2}}$ denotes the Hausdorff measure of noncompactness in the space $L^{2}(\R^{N})$.
\end{lemma}
\begin{proof}
Observe that by \eqref{ine11}, we have
\begin{equation}\label{rnon1}
\|S_{A-\lambda}(t)u\|_{L^{2}} \le e^{-\gamma t} \|u\|_{L^{2}}, \quad t\ge 0, \ u\in X_+^{e}.
\end{equation}
Let $W \subset X$ be any bounded set and $t \ge 0$. Then, by the standard properties of the measure $\beta_{L^2}$ (see \cite{MR1189795}), we have
$$
\beta_{L^{2}}(S_{A-\lambda}(t) W ) \le \beta_{L^{2}}(S_{A-\lambda}(t) (I-Q_+^{e}) W) 
+ \beta_{L^{2}}(S_{A-\lambda} (t)Q_+^{e} W) = \beta_{L^{2}}(S_{A-\lambda}(t)Q_{+}^{e} W),
$$
where the equality follows from the fact that the bounded set $S_{A-\lambda}(t) (I-Q_{+}^{e})W$ is contained in the finite-dimensional space $X_-\oplus X_0 \oplus X_{+}^{d}$. Then, in view of \eqref{rnon1} we infer that
$$
\beta_{L^{2}}(S_{A-\lambda}(t)Q_{+}^{e}W) \leq e^{-\gamma t} \beta_{L^{2}}(Q_{+}^{e}W).
$$
Using the fact that $\|Q_{+}^{e} \|_{\mathcal{L}(L^2, L^2)}=1$, we finally obtain
$$
\beta_{L^{2}}(Q_{+}^{e}W) \leq \beta_{L^2}(W),
$$
which combined with the previous inequalities proves \eqref{rownorm}, as desired.
\end{proof}

\section{\texorpdfstring{$L^1$ estimates for Schr\"odinger semigroup}
                          {L1 estimates for Schrodinger semigroup}}
A natural approach to extending the Laplace operator to other Lebesgue spaces is through convolution with the heat kernel
\[
K(t,x) := \frac{1}{(4\pi t)^{N/2}} \exp\Big(-\frac{|x|^2}{4t}\Big).
\]
Then the family of operators $\{S(t)\}_{t \ge 0}$ given by
\begin{equation}\label{heat-semi}
(S(t)u)(x) := \int_{\mathbb{R}^N} K(t,x-y)\, u(y)\, dy
\end{equation}
defines a Gaussian $C_0$-semigroup on $L^p(\mathbb{R}^N)$ for any $1 \le p < \infty$ (see, e.g., \cite[p.~150]{MR2798103}). We denote by $A_{0,p}$ the generator of this semigroup, which provides a realization of the $N$-dimensional Laplace operator in $L^p(\mathbb{R}^N)$. In particular, $A_{0,2} = A_0$. We recall that using the characterization of the Kato class \cite[Prop.~A.2.3]{MR0670130}, for any $\varepsilon > 0$ there exists a constant $C_\varepsilon > 0$ such that
\begin{equation}\label{ineq-l1}
\|Vu\|_{L^1} \leq \varepsilon \|A_{0,1} u\|_{L^1} + C_\varepsilon \|u\|_{L^1}, \quad u \in D(A_{0,1}),
\end{equation}
where $A_{0,1}$ denotes the generator of the semigroup \eqref{heat-semi} on $\tilde X := L^{1}(\mathbb{R}^{N})$. Since $A_{0,1}$ is a sectorial operator (see \cite[p.~150]{MR2798103} and \cite[Th.~3.7.11]{MR2798103}), the inequality \eqref{ineq-l1}, together with the perturbation result \cite[Prop.~1.3.2]{dlotko-cholewa}, implies that $\tilde A := A_{0,1} + V$ is also sectorial. 
By \cite[Th. 1]{MR0836002}, we have $\sigma(\tilde A) = \sigma(A) \subset \R$. Moreover, from the proof of \cite[Th. 3.1]{MR0869525}, it follows that for every $\mu \in \varrho(A) = \varrho(\tilde A)$ and every $u \in L^{1}(\mathbb{R}^{N}) \cap L^{2}(\mathbb{R}^{N})$, the resolvents of $A$ and $\tilde A$ coincide in $u$, that is,
\begin{equation}\label{res-eq}
(\mu I - \tilde A)^{-1}u = (\mu I - A)^{-1}u.
\end{equation}
Consequently, using the Euler formula for $C_{0}$-semigroups (see \cite[Th.~1.8.3]{Pazy}), we obtain
\begin{equation}\label{eq-sem-1}
S_{A}(t) u = S_{\tilde A}(t) u, \quad u \in L^{1}(\mathbb{R}^{N}) \cap L^{2}(\mathbb{R}^{N}), \ t \ge 0.
\end{equation}
The following proposition establishes the $L^{1}$ estimates for the semigroup generated by $A$.
\begin{proposition}\label{prop-est-l1}
If $\inf\sigma_{ess}(A) > \lambda$, then there exist constants $C_{Q} >0$ and $\delta_{Q} > 0$ such that
\begin{align*}
\|Q_{+}S_{A-\lambda}(t) u\|_{L^{1}} \le C_{Q} e^{-\delta_{Q} t} \|Q_{+}u\|_{L^{1}}, \quad t \ge 0,
\end{align*}
for $u\in L^{2}(\R^{N})\cap L^{1}(\R^{N})$.
\end{proposition}
\begin{proof}
Since $A$ is self-adjoint, the part of its spectrum $(-\infty,\mu_{e})\cap\sigma(A)$, where $\mu_{e}:=\inf\sigma_{\mathrm{ess}}(A)$, consists of isolated eigenvalues with finite multiplicity. Moreover, from \cite[Th. 1]{MR0836002} it follows that the part of the spectrum in the half-plane $\{\mathrm{Re}\,z < \mu_e\}$ is identical for both $A$ and $\tilde A$ allowing us to define
$$\sigma^{\lambda}_{-} := \sigma(A)\cap \{\mathrm{Re}\,z\le \lambda\} = \sigma(\tilde A)\cap\{\mathrm{Re}\,z\le \lambda\}.$$ 
For each $\mu \in \sigma_{-}^{\lambda}$, let $\tilde{P}_{\mu}$ be the restriction to $\tilde{X}$ of the Riesz projection defined on $\tilde{X}_{\mathbb{C}}$ by formula \eqref{riesz-proj}. Analogously, we denote by $P_{\mu}$ the corresponding projection on the Hilbert space $X$ associated with the same eigenvalue.
By the consistency of the resolvents \eqref{res-eq}, these projections satisfy 
\begin{equation}\label{eq-p}
P_{\mu}u = \tilde{P}_{\mu}u, \quad u \in L^1(\R^{N}) \cap L^2(\R^{N}).
\end{equation}
Let us denote $\tilde{X}_\mu = \tilde P_{\mu}(\tilde X)$. We claim that $\tilde{X}_\mu = \ker(\mu I - \tilde{A})$. From \cite[Prop. 6.3]{MR1361167} it follows that $\ker(\mu I - \tilde{A}) \subset \tilde{X}_\mu$. To verify the reverse inclusion, let $u \in \tilde{X}_\mu$. By a density argument, there exists a sequence $(u_n) \subset L^1(\mathbb{R}^N) \cap L^2(\mathbb{R}^N)$ such that $u_n \to u$ in $L^1(\mathbb{R}^N)$. The continuity of the projection $\tilde P_{\mu}$ and \eqref{eq-p} imply that
\begin{equation}
    P_\mu u_n = \tilde{P}_\mu u_n \to \tilde{P}_\mu u = u \quad \text{in } L^1(\mathbb{R}^N).
\end{equation}
Since each $P_\mu u_n$ belongs to the finite-dimensional space $\ker(\mu I - A)$, which is necessarily closed in $L^1(\mathbb{R}^N)$, it follows that the limit $u$ belongs to $\ker(\mu I - A)$. Since $\ker(\mu I - A) \subset \ker(\mu I - \tilde{A})$, we conclude that $u \in \ker(\mu I - \tilde{A})$, as desired.

By \cite[Th. 1]{MR0869525}, each eigenvalue in $\sigma^{\lambda}_{-}$ has the same geometric and algebraic multiplicities for $A$ and $\tilde A$, and the corresponding eigenspaces coincide. Hence, $\tilde X_{0} := \ker(\tilde A - \lambda) = X_{0}$ and 
\begin{equation}\label{eq-sets}
\tilde X_{-} := \bigoplus_{\mu\in \sigma^{\lambda}_{-}\setminus\{\lambda\}}\mathrm{Ker}\,(\tilde A - \mu) = \bigoplus_{\mu\in \sigma^{\lambda}_{-}\setminus\{\lambda\}}\mathrm{Ker}\,(A - \mu) = X_{-}.
\end{equation}
If we denote $\sigma^{\lambda}_{+}:=\sigma(\tilde A)\setminus\sigma^{\lambda}_{-}$, then applying the spectral theorem (see Appendix), we obtain a direct sum decomposition $L^{1}(\R^{N})=\tilde X_- \oplus \tilde X_0 \oplus \tilde X_+$ into closed subspaces of $L^1(\R^{N})$ that are invariant under $A_{1}$, and 
$$\sigma (\tilde A\,|\,\tilde X_-) = \sigma^{\lambda}_{-}\setminus\{\lambda\}\subset\R, \quad  \sigma (\tilde A\,|\,\tilde X_+) = \sigma^{\lambda}_{+}.$$ 
Moreover, since $\tilde A$ is sectorial, by \cite[Th. 1.5.3]{Henry}, there exist $C_{Q}>0$ and $\delta_{Q} > 0$ such that 
\begin{equation}\label{ineq-sem-1}
\|S_{\tilde A-\lambda} (t)u\|_{L^1}\leq C_{Q} e^{-\delta_{Q} t} \| u\|_{L^{1}}\quad   \mbox{ for} \ \ u\in \tilde X_+, \ t\geq 0.
\end{equation}
Let us denote by
\begin{equation*}
    \tilde{Q}_{-} := \sum_{\mu \in \sigma^{\lambda}_{-} \setminus \{\lambda\}} \tilde{P}_{\mu}, \quad \tilde{P} := \tilde{P}_{\lambda}, \quad \text{and} \quad \tilde{Q}_{+} := I - \tilde{Q}_{-} - \tilde{P}
\end{equation*}
the projections in $L^1(\mathbb{R}^N)$ onto the subspaces $\tilde{X}_{-}$, $\tilde{X}_0$, and $\tilde{X}_{+}$, respectively.
In view of \eqref{eq-sets}, we have $X_{-}=\tilde X_{-}$ and $X_{0}=\tilde X_{0}$, which implies 
$$
Q_- u = \tilde Q_- u,  \  \  \  Pu = \tilde P u \ \mbox{ and }   \  Q_+ u= \tilde Q_+ u \quad\text{for} \ \  u\in L^1(\R^N) \cap L^2(\R^N).
$$
Combining this with  \eqref{eq-sem-1}, for any $u\in L^1(\R^N)\cap L^2(\R^N)$, we obtain
$$
Q_+ S_{A-\lambda}(t)u = S_{A-\lambda}(t) Q_+ u = S_{A-\lambda}(t) \tilde Q_+ u = S_{\tilde A-\lambda}(t) \tilde Q_+ u, \quad t \ge 0,
$$
which, together with \eqref{ineq-sem-1}, gives
$$
\|Q_+ S_{A-\lambda}(t)u \|_{L^{1}}\leq C_{Q} e^{-\delta_{Q} t} \| \tilde Q_+ u\|_{L^1} = C_{Q} e^{-\delta_{Q} t}\| Q_+ u\|_{L^1}, \quad t \ge 0,
$$
concluding the proof.
\end{proof}
		
\section{\texorpdfstring{$L^\infty$ estimates for projections of $H^1$-bounded solutions}{L infinity-estimates of H1-bounded solutions}}\label{sec-3}
Let us consider the differential equation
\begin{equation}\label{main-eq}
\dot u (t) = -A u(t) + \lambda u(t) + F(u(t)), \quad t>0,
\end{equation}
where $\lambda$ is a real number, $A =-\Delta+V$ is the Schr\"{o}dinger operator on the space $X=L^2(\R^N)$ defined in Section \ref{sec-2}, and $F:H^1(\R^N) \to L^2(\R^N)$ is a nonlinear map satisfying the Lipschitz condition
\begin{equation}\label{F-lip-cond}
\|F( u_{1})-F( u_{2})\|_{L^2}\leq L \|u_{1}-u_{2}\|_{H^1} \quad  \text{ for all } \ \ u_{1}, u_{2}\in H^1(\R^N),\\
\end{equation}
for some constant $L>0$. It is well known (see, e.g., \cite[Th. 3.3.3]{Henry}) that for every 
$u_0 \in X^{1/2} = H^1(\mathbb{R}^N)$, there exists a unique global solution $u=u(\,\cdot\,; u_{0}):[0,\infty) \to H^1(\R^N)$ of the equation \eqref{main-eq} such that $u(0)= u_{0}$ and 
\begin{equation}\label{sol-reg}
u\in C([0,+\infty), H^1(\R^N))\cap C((0,+\infty), D(A)) \cap C^1((0,+\infty), L^2(\R^N)).
\end{equation}
		Consequently, the equation \eqref{main-eq} determines the semiflow $\{\Phi_t\}_{t\geq 0}$ on the space $H^1(\R^N)$, given by
		\begin{equation}\label{Phi-par-sem-def}
			\Phi(t, u_{0}) := u (t; u_{0}), \quad u_{0}\in H^1(\R^N), \ t\ge 0.
		\end{equation}
In this section, our aim is to establish estimates for full solutions for the semiflow $\{\Phi_{t}\}_{t\ge 0}$ under the assumption that the mapping $F$ is $L^\infty$-bounded, that is, there exists a constant $M>0$ such that
\begin{equation}\label{F-L-infty-bdd}
\|F( u) \|_{L^{\infty}} \le M, \quad \mbox{ for all } u \in H^{1}(\R^{N}).
\end{equation}
		\begin{proposition}
			Let us assume that $u\in C(\R,H^{1}(\R^{N}))$ is a solution of \eqref{main-eq}. Then $u(t)\in L^{\infty}(\R^{N})\cap C(\R^{N})$ for all $t\in\R$ and $u\in L^{\infty}_{\mathrm{loc}}(\R;L^{\infty}(\R^{N}))$. 
		\end{proposition}
		\begin{proof}
Observe that any full solution $u$ of the semiflow $\{\Phi_{t}\}_{t\ge0}$ satisfies the Duhamel formula
		\begin{align}\label{duhamel_eq}
			u(t) = S_{A-\lambda}(t-t')u(t') + \int_{t'}^{t} S_{A-\lambda}(t-\tau)F(u(\tau))\,d\tau
		\end{align}
for $t,t'\in \R$ with $t>t'$. By the inequalities \eqref{est-2} and \eqref{F-L-infty-bdd}, we have
			\begin{equation}
				\begin{aligned}\label{est-s-g}
					\|S_{A-\lambda}(t-\tau)F(u(\tau))\|_{L^\infty} & \leq C e^{(t-\tau)(\omega+\lambda)} \|F ( u(\tau))\|_{L^\infty} \leq  C M e^{(t-\tau)(\omega+\lambda)}
				\end{aligned}
			\end{equation}
			for all $\tau < t$. Let us define
			\begin{align*}
				I(t,t'):=\int_{t'}^{t}S_{A-\lambda}(t-\tau)F(u(\tau))\,d\tau.
			\end{align*}
			If we take $\varphi\in C^{\infty}_{c}(\R^{N})$, then applying \eqref{est-s-g} yields
			\begin{align*}
				\langle I(t,t'), \varphi \rangle_{L^{2}}
				& = \int_{t'}^{t}\langle S_{A-\lambda}(t-\tau) F(u(\tau)),\varphi\rangle_{L^{2}}\,d\tau\\
				& \le \int_{t'}^{t}\|S_{A-\lambda}(t-\tau) F(u(\tau))\|_{L^{\infty}}\|\varphi\|_{L^{1}}\,d\tau \\
				& \leq CM \|\varphi\|_{L^1} \int_{t'}^{t} e^{(t-\tau)(\omega+\lambda)} \,d\tau.
			\end{align*}
			This implies that $I(t,t')\in L^{1}(\R^{N})^{*}=L^{\infty}(\R^{N})$ and 
			\begin{align}\label{est-norm}
				\|I(t,t')\|_{L^\infty} \lesssim \int_{0}^{t-t'} e^{\tau(\omega+\lambda)} \,d\tau \ \mbox{ for all } \ t>t'.
			\end{align}
Let us take arbitrary $t \in \mathbb{R}$ and observe that for any $t > t'$ and sufficiently small $\varepsilon > 0$, we have
			\begin{align}\label{eq-linfty}
				u(t) = S_{A-\lambda}(t-t')u(t') + S_{A-\lambda}(\ve)I(t-\ve,t') + I(t,t-\ve).
			\end{align}
	Since $u(t')\in L^{2}(\R^{N})$ and $I(t-\ve,t')\in L^{2}(\R^{N})$, it follows from \eqref{est-3} and \eqref{est-1-a} that 
	$$
	S_{A-\lambda}(t-t')u(t') + S_{A-\lambda}(\ve)I(t-\ve,t') \in L^{\infty}(\R^{N})\cap C(\R^{N}).
	$$
	Combining this with \eqref{eq-linfty} and the fact that $I(t-\ve,t')\in L^{\infty}(\R^{N})$, gives $u(t)\in L^{\infty}(\R^{N})\cap C(\R^{N})$ as desired. We now show that $u : I \to L^{\infty}(\mathbb{R}^n)$ is strongly measurable on every compact interval $I \subset \mathbb{R}$. To this end, let $t'\in\R$ be such that $I\subset (t'+1, \infty)$. By the estimate \eqref{est-norm}, for sufficiently small $\ve>0$, we have
			\begin{equation}
				\begin{aligned}\label{eq-conv-1}
					\|u(t) - S_{A-\lambda}(t-t')u(t') - S_{A-\lambda }(\ve)I(t-\ve,t')\|_{L^\infty} = \|I(t,t-\ve)\|_{L^\infty} \lesssim 
					\int_{0}^{\ve} e^{s(\omega+\lambda)} \,ds
				\end{aligned}
			\end{equation}
for $t\in I$. Now let us note that the map $g:I\to L^2(\R^N)$  
			$$
			g(t) = S_{A-\lambda}(t-\ve-t')u(t') + I(t-\ve,t'), \quad t\in I,
			$$ is well-defined and continuous, and hence strongly measurable.
			Combining this with the inequality \eqref{est-3}, we deduce that  
			\begin{align*}
				S_{A-\lambda}(t-t')u(t') + S_{A-\lambda} (\ve)I(t-\ve,t') = S_{A-\lambda}(\ve)g(t), \quad t\in I,
			\end{align*}
is strongly measurable as an $L^\infty(\mathbb{R}^{N})$-valued function.
			Therefore, letting $\varepsilon\to 0$ in \eqref{eq-conv-1}, we conclude that $u:I\to L^\infty(\mathbb{R}^{N})$
is strongly measurable, being the pointwise limit of strongly measurable functions. By the inequalities \eqref{est-1} and \eqref{est-3}, we have, for $t\in I$, we have
			\begin{equation}
				\begin{aligned}\label{eq-ineq-1}
					\|S_{A-\lambda}(t-t')u(t')\|_{L^\infty} & = \|S_{A-\lambda}(t-t'-1/2)S_{A-\lambda}(1/2)u(t')\|_{L^\infty} \\
					&\lesssim (t-t'-1/2)^{-N/4}e^{(\omega+\lambda)(t-t'-1/2)} \|S_{A-\lambda}(1/2)u(t')\|_{L^{2}} \\
					&\lesssim 2^{N/4}e^{(\omega+\lambda)(t-t'-1/2)}e^{(\omega+\lambda)/2}\|u(t')\|_{L^{2}} \\
					&= 2^{N/4}e^{\omega+\lambda}\|u(t')\|_{L^{2}}.
				\end{aligned}
			\end{equation}
			Applying the estimates \eqref{eq-ineq-1} and  \eqref{est-norm} to the right-hand side components of 
			\eqref{duhamel_eq} implies that $u\in L^{\infty}(I;L^{\infty}(\R^{N}))$, which completes the proof.
		\end{proof}
		
		Now we shall look for $L^\infty$ estimates of bounded solutions of \eqref{main-eq}, assuming that $\lambda < \inf\sigma_{ess}(A)$. Let us recall that $\gamma>0$ be an arbitrary number such that \eqref{del-alph-zer} holds and $Q_{\pm}$, $Q_{+}^{d}$ are the projections defined in Section \ref{sec-2}.
		
		\begin{proposition}\label{prop-est-infty-1}
			There exists a constant $R_{Q,\infty} >0$, depending on the potential $V$ and the constant $M>0$ from condition \eqref{F-L-infty-bdd}, such that for any full solution $u:\R\to H^1(\R^N)$ of the equation \eqref{main-eq}, that is bounded in the space $H^{1}(\R^{N})$, the following estimate holds
			\begin{align*}
				\|Q_{+}u(t)\|_{L^{\infty}} \le R_{Q,\infty}, \quad t\in\R.
			\end{align*}
		\end{proposition}
		\begin{proof}
			Let us take $t,t'\in\R$ such that $t>t'+1$.
			Acting on both sides of the formula \eqref{duhamel_eq} with the projection $Q_{+}$, we obtain
			\begin{align}\label{duhamel_eqa}
				Q_{+}u(t) = S_{A-\lambda} (t-t')Q_{+} u(t') + Q_{+}I(t,t').
			\end{align}
Then \eqref{est-3} and \eqref{est-11aa}, give 
			\begin{equation}
				\begin{aligned}\label{eq-ineq-111}
					\|S_{A-\lambda} (t-t')Q_{+}u(t')\|_{L^\infty} & = \|S_{A-\lambda}(1)S_{A-\lambda} (t-t'-1) Q_{+}u(t')\|_{L^\infty} \\
					&\le  Ce^{(\lambda + \omega)} \|S_{A-\lambda} (t-t'-1)Q_{+}u(t')\|_{L^2} \\
					&\le C e^{(\lambda + \omega)}e^{-\delta_{+}(t-t'-1)}\|Q_{+}u(t')\|_{L^2}.
				\end{aligned}
			\end{equation}
Let $\varphi\in C_c^{\infty}(\R^{N})$ be an arbitrary test function. Since $Q_{+}$ and $S_{A-\lambda} (t)$ for $t\ge 0$, are self-adjoint operators on $L^2(\R^N)$, we have
			$$
			\begin{aligned}\label{est-m1}
				\langle Q_{+}I(t,t'),\varphi\rangle_{L^{2}} & = \int_{t'}^{t} \langle Q_{+}S_{A-\lambda}(t-\tau)F(u(\tau)),\varphi\rangle_{L^{2}}\,d\tau \\
				& = \int_{t'}^{t} \langle S_{A-\lambda}(t-\tau)F(u(\tau)),Q_{+}\varphi \rangle_{L^{2}}\,d\tau \\
				& = \int_{t'}^{t} \langle F(u(\tau)),S_{A-\lambda}(t-\tau)Q_{+}\varphi \rangle_{L^{2}}\,d\tau.
			\end{aligned}
			$$
Hence, by applying Proposition \ref{prop-est-l1} and Proposition \ref{prop-inclusion} with $p=1$, we obtain
			\begin{equation*}
				\begin{aligned}
					\langle Q_{+}I(t,t'),\varphi\rangle_{L^2} & \le \int_{t'}^{t}\|F(u(\tau))\|_{L^\infty}\|S_{A-\lambda}(t-\tau)Q_{+}\varphi\|_{L^{1}}\,d\tau \\
					& \le M C_Q \int_{t'}^{t}e^{-\delta_{Q}(t-\tau)}\|Q_{+}\varphi\|_{L^{1}}\,d\tau \\
					& \le MC_Q \delta_{Q}^{-1}\|Q_{+}\|_{\mathcal{L}(L^{1},L^{1})}\left(1- e^{-\delta_{Q}(t-t')}\right)\|\varphi\|_{L^{1}}.
				\end{aligned}
			\end{equation*}
			Consequently, by the the duality argument, 
			\begin{align}\label{est-ineq-2}
				\|Q_{+}I(t,t')\|_{L^\infty} \le MC_Q\delta_{Q}^{-1}\|Q_{+}\|_{\mathcal{L}(L^{1},L^{1})}\left(1- e^{-\delta_{Q}(t-t')}\right).
			\end{align}
			Using \eqref{eq-ineq-111} and \eqref{est-ineq-2} together with \eqref{duhamel_eqa}, one gets
			\begin{align*}
				\|Q_{+}u(t)\|_{L^\infty} \leq C e^{(\lambda + \omega)} e^{-\delta_{+}(t-t'-1)}\|Q_{+}u(t')\|_{L^{2}} + MC_{Q}\delta_{Q}^{-1}\|Q_{+}\|_{\mathcal{L}(L^{1},L^{1})}\left(1- e^{-\delta_{Q}(t-t')}\right).
			\end{align*}
Letting $t' \to -\infty$ and using the boundedness of the solution $u$ in $L^2(\mathbb{R}^N)$, we conclude that
			\begin{align*}
				\|Q_{+}u(t)\|_{L^\infty} \le M C_{Q}\delta_{Q}^{-1}\|Q_{+}\|_{\mathcal{L}(L^{1},L^{1})} \ \text{ for all }  t\in\R,
			\end{align*}
			which completes the proof of the proposition.
		\end{proof}		
We will also study estimates for the projections of bounded solutions onto $X_-$ and $X_{+}^{d}$.
		\begin{proposition}\label{prop-est-infty-2}
			There exists $R_{d}>0$, depending only on the potential $V$ and the constant $M>0$ from condition \eqref{F-L-infty-bdd}, such that for any full solution $u:\R\to H^1(\R^N)$ of the equation \eqref{main-eq}, that is bounded in $H^{1}(\R^{N})$, the following estimates hold
			\begin{align}\label{eq-p-1}
				\|Q_{-}u(t)\|_{L^{2}} + \|Q_{-}u(t)\|_{L^{\infty}} & \le R_d,\\ \label{eq-p-2}
				\|Q_{+}^{d}u(t)\|_{L^{2}} + \|Q_{+}^{d}u(t)\|_{L^{\infty}} & \le R_d,
			\end{align}
			for all $t\in\R$.
		\end{proposition}
		\begin{proof}
Since $X_{+}^{d}$ is finite-dimensional, combining Proposition \ref{prop-inclusion} with $p=\infty$ and Proposition \ref{prop-est-infty-1}, gives
			\begin{align*}
				\|Q_{+}^{d}u(t)\|_{L^{2}} + \|Q_{+}^{d}u(t)\|_{L^{\infty}} \lesssim \|Q_{+}^{d}u(t)\|_{L^{\infty}} = \|Q_{+}^{d}Q_{+}u(t)\|_{L^{\infty}} 
				 \lesssim \|Q_{+}u(t)\|_{L^{\infty}} \le R_{Q,\infty}
			\end{align*}
			for all $t \in \mathbb{R}$, which gives inequality \eqref{eq-p-2}. To obtain \eqref{eq-p-1}, we observe that the semigroup $\{S_A (t)\}_{t\geq 0}$ can be extended to a $C_0$-group on $X_{-}$. We then apply the operator $Q_{-}$ to the Duhamel formula \eqref{duhamel_eq} and arrive at
			\begin{align}\label{eq-s-1b}
				Q_{-}u(t') = S_{A-\lambda}(t'-t)Q_{-}u(t) - \int_{t'}^{t}S_{A-\lambda} (t'-\tau)Q_{-}F(u(\tau))\,d\tau, \quad t > t'.
			\end{align}
In view of the inequality \eqref{ine33}, one has
			\begin{equation}
				\begin{aligned}\label{eq-s-1c}
					\|Q_{-}u(t')\|_{L^{2}} & \le \|S_{A-\lambda}(t'-t)Q_{-}u(t)\|_{L^{2}} + \int_{t'}^{t}\|S_{A-\lambda}(t'-\tau)Q_{-}G(u(\tau),s)\|_{L^{2}}\,d\tau \\
					&\le C_- e^{(t'-t)\delta_{-}}\|Q_{-}u(t)\|_{L^{2}} + \int_{t'}^{t}C_- e^{(t'-\tau)\delta_{-}}\|Q_{-}F(u(\tau))\|_{L^{2}}\,d\tau.
				\end{aligned}
			\end{equation}
			Since $X_{-}$ is finite-dimensional, Proposition \ref{prop-inclusion} with $p=\infty$ yields
			\begin{equation}\label{ineq-f}
					\|Q_{-}F(u(\tau))\|_{L^{2}} \lesssim \|Q_{-}F(u(\tau))\|_{L^{\infty}} 
					\lesssim \|F(u(\tau))\|_{L^{\infty}} \le M, \quad \tau\in\R,
			\end{equation}
			where the last inequality follows from the assumption \eqref{F-L-infty-bdd}. Combining \eqref{eq-s-1c} and \eqref{ineq-f} then gives
			\begin{align*}
				\|Q_{-}u(t')\|_{L^{2}} & \lesssim e^{(t'-t)\delta_{-}}\|Q_{-}u(t)\|_{L^{2}} + M \int_{t'}^{t}e^{(t'-\tau)\delta_{-}}\,d\tau \\
				& \le e^{(t'-t)\delta_{-}}\|u(t)\|_{L^{2}} + M \delta_{-}^{-1}\left(1-e^{(t'-t)\delta_{-}}\right).
			\end{align*} 
Passing to the limit $t \to \infty$ in the above inequality and using the boundedness of $u$ in $H^1(\mathbb{R}^N)$, we deduce that
			\begin{align*}
				\|Q_{-}u(t')\|_{L^{2}} + \|Q_{-}u(t')\|_{L^{\infty}}\lesssim \|Q_{-}u(t')\|_{L^{2}} \lesssim M \delta_{-}^{-1} \lesssim 1,\quad t'\in\R,
			\end{align*}
which yields estimate~\eqref{eq-p-1} and completes the proof.
		\end{proof}

\section{\texorpdfstring{$H^1$ estimates for projections of $H^1$-bounded solutions}{H1 estimates of H1-bounded solutions}}
Consider equation \eqref{main-eq}, where the operator $A$ and the map $F$ are as in Section \ref{sec-3}. Suppose that $\lambda < \inf \sigma_{\mathrm{ess}}(A)$, and fix $\gamma > 0$ such that \eqref{del-alph-zer} holds. We assume that 
\begin{equation}\label{pseudo-lip-ineq}
\|F(u)\|_{L^2}\leq L_\sigma \left( \|P u\|_{L^2} + \|Q u\|_{L^2}\right) + C_\sigma  \quad  \mbox{ for all } \ \ u\in H^1(\R^N),
\end{equation}
where $L_\sigma>0$ and $C_\sigma>0$ are constants, and that 
		\begin{equation}\label{pseudo-lip-vs-spectrum}
			L_\sigma < \gamma.
		\end{equation} 
We begin by establishing $L^2$ estimates for projections of bounded solutions.
\begin{proposition}\label{prop-est-q}
There exists $R_e>0$, depending on the potential $V$ and the constants $M$, $L_\sigma$, $C_{\sigma}$ such that for any full solution $u$ of \eqref{main-eq} bounded in $H^{1}(\R^{N})$, we have
			\begin{align}\label{est-m-1}
				\|Q_{+}^{e}u(t)\|_{L^{2}} \le R_{e} \left(1+\sup_{\tau\in\R} \|Pu(\tau)\|_{L^2}\right) \ \text{ for all } \ t\in\R.
			\end{align}
		\end{proposition}
		\begin{proof} 
			By Proposition \ref{prop-est-infty-2}, there exists $R_d > 0$ such that 
			$$
			\|Q_- u(t)\|_{L^2} + \|Q_{+}^{d} u(t)\|_{L^{2}} \leq R_d \ \mbox{ for all } \ t\in \R,
			$$
			for any bounded solution $u:\R\to H^1(\R^N)$ of \eqref{main-eq}. \\
			\indent Fix such a solution $u$ and apply $Q_+^{e}$ to both sides of the formula \eqref{duhamel_eq}. Using \eqref{ine11}, we obtain 
		\begin{equation}\label{eq-q-1}
			\|Q_{+}^{e} u(t)\|_{L^{2}}\le e^{-\gamma (t-t')}\|Q_{+}^{e}u(t')\|_{L^{2}} + \int_{t'}^{t}e^{-\gamma (t-\tau)}\|Q_{+}^{e}F(u(\tau))\|_{L^{2}}\,d\tau.
		\end{equation}
			In view of  \eqref{pseudo-lip-ineq}, for all $\tau \in\R$, 
			\begin{equation}
				\begin{aligned}\label{eq-q-a2}
					&\|Q_{+}^{e}F(u(\tau) )\|_{L^{2}} \leq  \|F(u(\tau))\|_{L^{2}} \le L_\sigma\left( \|Pu(\tau)\|_{L^{2}}+
					\|Q u(\tau)\|_{L^{2}}\right) + C_\sigma \\
					&\qquad \le  L_\sigma \|Q_{+}^{e}u(\tau)\|_{L^{2}} +  L_\sigma (\|Q_{-}u(\tau)\|_{L^{2}} + \|Q_{+}^{d}u(\tau)\|_{L^{2}} + \|Pu(\tau)\|_{L^2} ) + C_\sigma \\
					&\qquad \le  L_\sigma \|Q_{+}^{e}u(\tau)\|_{L^{2}} +  L_\sigma R_d + C_\sigma + L_\sigma \|Pu(\tau)\|_{L^2}.
				\end{aligned}
			\end{equation}
			Consequently,
			$$
			\|Q_{+}^{e}F(u(\tau) )\|_{L^{2}}  \leq L_\sigma \|Q_{+}^{e}u(\tau)\|_{L^{2}} + R_{\sigma},
			$$
			where 
			$$
			R_{\sigma} := L_\sigma R_d + C_\sigma + L_\sigma \sup_{\tau\in\R} \|Pu(\tau)\|_{L^2}. 
			$$
			Combining this with \eqref{eq-q-1}, gives
			$$
			e^{\gamma t} \|Q_{+}^{e}u(t)\|_{L^2} - e^{\gamma t'}\|Q_{+}^{e} u(t')\|_{L^2}\leq L_\sigma \int_{t'}^{t}
			e^{\gamma \tau} \|Q_{+}^{e}u(\tau)\|_{L^2} d\tau + R_{\sigma} \int_{t'}^{t} e^{\gamma\tau}d\tau.
			$$			
			Hence, for $\eta: \R\to \R$ defined by 
			$$
			\eta(t) := e^{\gamma t} \|Q_+^e u(t) \|_{L^2}, \ t\in\R,
			$$
			we have, for any $t,t'\in\R$ with $t>t'$,
			\begin{equation}
				\begin{aligned}\label{eq-q-3}
					\eta(t) - \eta(t') \le L_\sigma \int_{t'}^{t} \eta (\tau)\,d\tau + R_{\sigma} \int_{t'}^{t}e^{\gamma \tau}\,d\tau,\quad t'<t.
				\end{aligned}
			\end{equation}
By the regularity \eqref{sol-reg}, the function $\eta$ is continuously differentiable in a neighborhood of any point $t \in \mathbb{R}$ with $Q_+^e u(t) \neq 0$. Moreover, 
			\begin{equation}\label{bdd-eta-prop}
				\sup_{t \in\R} e^{-\gamma t} \eta(t) < +\infty,
			\end{equation}
			since $u$ is bounded in $H^1(\mathbb{R}^N)$. \\
			\indent We claim that 
			\begin{equation} \label{Qe-bounded-lemma}
				\eta(t)\leq R_{\sigma} e^{\gamma t}/(\gamma -L_\sigma) \ \ \mbox{ for all } \ \ t\in\R.
			\end{equation}
			Indeed, suppose to the contrary that there is $t_0\in \R$ such that
			\begin{equation}\label{eq-c-1}
				\eta(t_0) > R_{\sigma} e^{\gamma t_0}/(\gamma -L_\sigma).
			\end{equation}
			Consider the set 
			\begin{align*}
				I_0 := \left\{t\le t_{0} \ | \ \eta(t')\ge R_{\sigma} e^{\gamma t'}/(\gamma-L_\sigma) \ \text{ for all } \ t'\in [t,t_{0}]\right\}.
			\end{align*}
			We claim that $I_0 = (-\infty,t_{0}]$. Otherwise, let $\bar t := \inf I_0 >-\infty$. By continuity, there exists $\rho >0$ such that $Q_+^eu(t')\neq 0$ for $t'\in (\bar t-\rho,\bar t]$. Dividing both sides of \eqref{eq-q-3} by $t-t'>0$ and passing to the limit $t'\to t$ gives
			\begin{align*}
				\eta' (t) \le L_\sigma \eta(t)+ R_{\sigma} e^{\gamma t} , \quad t\in (\bar t-\rho, \bar t].
			\end{align*}
			Rewriting the above inequality as
			\begin{align*}
				(\eta(t) e^{- L_\sigma t} )' \le R_{\sigma} e^{(\gamma - L_\sigma)t}
			\end{align*}
			and integrating, we deduce that 
			\begin{align}\label{eq-11}
				e^{-L_\sigma t} \left( \eta (t) - R_{\sigma} e^{\gamma t}/(\gamma-L_\sigma) \right) \le e^{-L_\sigma t'} \left( \eta(t') -  R_{\sigma} e^{\gamma t'}/(\gamma-L_\sigma)\right)
			\end{align}
			for any $t,t'\in (\bar t -\rho, \bar t]$ with $t>t'$. 
			Since $I_0$ is closed, we have $\bar t\in I_0$ and we can use \eqref{eq-11} to obtain  
			\begin{align}
				0\le e^{-L_\sigma (\bar t-t')}\left(\eta( \bar t) - R_{\sigma} e^{\gamma \bar t}/(\gamma-L_\sigma) \right) 
				\le \eta(t') - R_{\sigma} e^{\gamma t'}/(\gamma-L_\sigma)
			\end{align}
			for $t'\in(\bar t-\rho,\bar t]$, which contradicts the definition of $\bar t=\inf I_0$. Hence $I_0 = (-\infty, t_0]$, as claimed. 
			Note that \eqref{eq-11} then holds for all $t,t'\in (-\infty,t_{0}]$ with $t>t'$. Consequently,
			\begin{align*}
				e^{(\gamma-L_\sigma)t}\left(e^{-\gamma t} \eta (t) - R_{\sigma}/(\gamma- L_\sigma) \right) \le e^{(\gamma-L_\sigma)t'} \left( e^{-\gamma t'}\eta(t') - R_{\sigma}/(\gamma-L_\sigma) \right).
			\end{align*}
			Taking into account \eqref{bdd-eta-prop} and passing to the limit as $t'\to-\infty$, we find that 
			\begin{align*}
				e^{(\gamma-L_\sigma)t}\left(e^{-\gamma t} \eta(t) - R_{\sigma}/(\gamma-L_\sigma) \right) \le 0,\quad t\le t_{0},
			\end{align*}
			which contradicts \eqref{eq-c-1} and completes the proof of \eqref{Qe-bounded-lemma}.\\
			\indent Now the assertion follows from \eqref{Qe-bounded-lemma} and the definitions of the function $\eta$ and the constant $R_{\sigma}$.
		\end{proof}		
\noindent We now use the $L^2$ estimates for projections of bounded solutions to derive the required $H^{1}$ estimates. 
		\begin{proposition} \label{Q-boundedness}
		There exists $R'_e>0$, depending on the potential $V$ and the constants $M$, $L_\sigma$, and $C_{\sigma}$, such that for any full solution $u$ of \eqref{main-eq} bounded in $H^{1}(\R^{N})$, the following estimate holds
\begin{align*}
\|Q_{+}^{e}u(t)\|_{H^{1}} \le R'_e \left( 1 + \sup_{\tau \in\R} \|P u(\tau )\|_{L^2}\right)\quad \mbox{ for all } t\in\R.
\end{align*}\end{proposition}
\begin{proof}
			Applying $Q_+$ to \eqref{duhamel_eq} and using \eqref{est-11bb}, we obtain
			\begin{equation}\label{ineq-11bb}
				\|Q_+^{e} u(t)\|_{H^{1}} \le C_+\frac{e^{-\delta_+ (t-t')}}{(t-t')^{1/2}} \|Q_+^{e} u(t')\|_{L^2}
				+ C_+ \int_{t'}^{t} \frac{e^{-\delta_{+}(t-\tau)}}{(t-\tau)^{1/2}} \|Q_+^{e} F(u(\tau))\|_{L^{2}} d\tau
			\end{equation}
			for all $t,t'\in\R$ such that $t>t'$. By \eqref{pseudo-lip-ineq} and \eqref{pseudo-lip-vs-spectrum}, for all $\tau\in\R$, we have
			\begin{equation}\label{eq-q-a2bbcc}
					\|Q_{+}^{e}F(u(\tau))\|_{L^{2}} \le \|F(u(\tau))\|_{L^{2}} 
						 \le L_\sigma \|Q u(\tau)\|_{L^{2}} + L_\sigma\| P u(\tau)\|_{L^2} + C_\sigma.
			\end{equation}
			Furthermore, in view of Propositions \ref{prop-est-infty-2} and \ref{prop-est-q}, there exists $R' >0$ such that 
			\begin{align}\label{eq-12-b}
				\|Q u(t)\|_{L^{2}} \le \|Q_{-}u(t)\|_{L^{2}} + \|Q^{d}_{+}u(t)\|_{L^{2}} + \|Q^{e}_{+}u(t)\|_{L^{2}}\le R' \left(1+\sup_{\tau\in\R} \|P u(\tau)\|_{L^{2}}\right)
			\end{align}
			for $t\in\R$. Combining this with \eqref{eq-q-a2bbcc}, gives $R''>0$ such that
			$$
			\|Q_{+}^{e}F(u(\tau))\|_{L^{2}} \leq R'' \left( 1+\sup_{s\in\R} \|Pu(s)\|_{L^2}\right) \ \mbox{ for all } \ \tau\in\R.
			$$
			Finally, since $u$ is bounded in $H^{1}(\R^{n} )$, passing to the limit $t' \to -\infty$ in \eqref{ineq-11bb} gives
			$$
			\|Q_+^{e} u(t)\|_{H^1} \leq C_+ R''  \left( 1 +\sup_{\tau\in\R} \|Pu(\tau)\|_{L^2} \right)\int_{-\infty}^{t}  \frac{e^{-\delta_+(t-\tau)}}{(t-\tau)^{1/2}} d\tau, \quad t\in\R,
			$$
			which completes the proof.
		\end{proof}
In summary, Propositions \ref{Q-boundedness} and \ref{prop-est-infty-2} yield the following.
		\begin{corollary}
					There exists $R'_e>0$, depending on the potential $V$ and the constants $M$, $L_\sigma$, $C_{\sigma}$ such that for any full solution $u$ of \eqref{main-eq} bounded in $H^{1}(\R^{N})$, the following estimate holds
			\begin{align}\label{eq-11-a}
				\|Q u(t)\|_{H^{1}} \le R_Q \left( 1 + \sup_{\tau \in\R} \|P u(\tau )\|_{L^2}\right)\quad \mbox{ for all } t\in\R.
			\end{align}
		\end{corollary}

\section{Compactness properties of families of semiflows}
In this section, we investigate the compactness properties of the family of semiflows associated with the equation
\begin{equation}\label{A-Fn-eqn}
\dot u (t) = - A u(t)+\lambda u(t)+F_n (u(t)), \quad t>0,
\end{equation}
where $\lambda$ is a real number, $A =-\Delta+V$ is the Schr\"{o}dinger operator defined in Section \ref{sec-2} and $F_n:H^1(\R^N) \to L^2(\R^N)$, $n\ge 1$, is a family of continuous nonlinearities satisfying the following uniform condensing condition: there exists $k < \inf \sigma_{ess} (-\Delta +V)-\lambda$ such that
\begin{equation} \label{k-set-for-seq}
\beta_{L^2} \left( \bigcup_{n\geq 1} F_n(W) \right) \leq k \beta_{L^2}(W) \quad \text{ for  any bounded } \ W\subset H^1(\R^N).
\end{equation}
The following proposition provides an exponential estimate for the measure of noncompactness along solutions of equation \eqref{A-Fn-eqn}.
\begin{proposition}\label{prop-con-poinc}
Let $u_n:[t_0,t_1]\to H^1(\mathbb{R}^N)$, $n\ge 1$, be a solution of equation \eqref{A-Fn-eqn}, where the family $\{F_n\}_{n\ge 1}$ satisfies the condition \eqref{k-set-for-seq}. Suppose that there exists $R>0$ such that
\begin{equation*}
\|u_{n}(t)\|_{H^1} \leq R \ \text{ for all } \ t\in [t_0, t_1], \ n\ge 1.
\end{equation*}
Then, for every $t \in [t_0,t_1]$, we have
\begin{align}\label{ineq-b-meas}
\beta_{L^{2}}(\{u_{n}(t)\}_{n\ge 1}) \le e^{-(\gamma - k)(t-t_{0})}\beta_{L^{2}}(\{u_{n}(t_{0})\}_{n\ge 1}),
\end{align}
where $\gamma$ is any number such that $k<\gamma < \inf \sigma_{ess}(A)-\lambda$ and $\lambda+\gamma\not\in\sigma(A)$.
\end{proposition}
In the proof, we use the following general property of the measure of noncompactness (see \cite{MR1189795, ObuKaZ}).
\begin{lemma}\label{integral-prop-of-measure} 
			Let $E$ be a separable Banach space and let $B\subset L^1([a,b],E)$ be a countable and integrably bounded set, i.e., there exists $c\in L^1([a,b])$ such that $$\|w(t)\|\leq c(t), \quad w\in B \text{ and a.e. }t\in [a,b].$$ Define the function $\phi (t):= \beta(\{u(t)|\, u\in B \})$ for $t\in[a,b]$. Then $\phi\in L^1([a,b])$ and
			$$
			\beta \left(\left\{
			\int_{a}^{b} u(\tau) \d \tau\,|\, u \in B\right\} \right)
			\leq \int_{a}^{b} \phi(\tau) \d \tau.
			$$
		\end{lemma}
		\begin{proof}[Proof of Proposition \ref{prop-con-poinc}]
			Applying the Hausdorff measure of noncompactness $\beta_{L^{2}}$ to the Duhamel formula, we obtain
			\begin{align*}
				\beta_{L^{2}}(\{u_{n}(t)\}_{n\ge 1}) & \le \beta_{L^{2}}(S_{A-\lambda}(t-t_{0})\{u_{n}(t_{0})\}_{n\ge 1}) + 
				\beta_{L^{2}}\left(\left\{\int_{t_{0}}^{t}S_{A-\lambda}(t-\tau)F_n(u_{n}(\tau)) d\tau \ | \ n\ge 1 \right\}\right) 
			\end{align*}
			for all $t\in [t_{0},  t_{1}]$. Then, using Lemmata \ref{lem-con-sem} and \ref{integral-prop-of-measure}, we obtain
			\begin{equation}
				\begin{aligned}\label{ineq-b-1}
					\beta_{L^{2}}(\{u_{n}(t)\}_{n\ge 1}) 
					 \le e^{-\gamma (t-t_{0})} \beta_{L^{2}}(\{u_{n}(t_{0})\}_{n\ge 1}) + \int_{t_{0}}^{t}e^{-\gamma(t-\tau)} \beta_{L^{2}}(\{F_n (u_{n}(\tau))\}_{n\ge 1})\,d\tau.
				\end{aligned}
			\end{equation}
Since the set $W_\tau := \{u_{n}(\tau) \ | \ n\ge 1\}$ is bounded for all $\tau\in[t_{0},t]$, the inequality \eqref{k-set-for-seq} implies
			\begin{align*}
				\beta_{L^{2}}(\{F_n(u_{n}(\tau)) \ | \ n\ge 1\}) \leq\beta_{L^{2}} \left(\bigcup_{n\geq 1} F_n(W_\tau)\right) \leq k \beta_{L^2} (W_\tau)=k\beta_{L^{2}}(\{u_{n}(\tau)\}_{n\ge 1}).
			\end{align*}
			Combining this with \eqref{ineq-b-1} gives
			\begin{equation*}
				\beta_{L^{2}}(\{u_{n}(t)\}_{n\ge 1}) \le e^{-\gamma (t-t_{0})}\beta_{L^{2}}(\{u_{n}(t_{0})\}_{n\ge 1}) + k \int_{t_{0}}^{t}e^{-\gamma (t-\tau)} \beta_{L^{2}}(\{u_{n}(\tau)\}_{n\ge 1})\,d\tau,
			\end{equation*}
			which, after an application of Gronwall's inequality, yields the desired inequality \eqref{ineq-b-meas}. 
		\end{proof}
		\begin{corollary} \label{cor-admissibility}
Assume that for each $n \ge 0$, the mapping $F_n$ satisfies the Lipschitz condition \eqref{F-lip-cond} with the same constant $L>0$. Suppose, moreover, that the family $\{F_n\}_{n \ge 1}$ satisfies \eqref{k-set-for-seq} with $k < \inf \sigma_{ess}(A)-\lambda$ and that
\[
F_n(u) \to F_0(u) \quad \text{in } \ L^2(\mathbb{R}^N), \ \ \text{as } \ n \to +\infty,
\]
for every $u \in H^1(\mathbb{R}^N)$. Let, for each $n \ge 1$, the mapping $u_n \colon [0,t_n] \to H^1(\mathbb{R}^N)$ be a solution of equation \eqref{A-Fn-eqn} with $t_n \to \infty$. If there exists $R > 0$ such that
\[
\|u_n(t)\|_{H^1} \le R \quad \text{for all } \ t \in [0, t_n] \ \text{ and } \ n \ge 1,
\]
then the sequence $\bigl(u_n(t_n)\bigr)$ contains a subsequence  converging in $H^1(\mathbb{R}^N)$.
		\end{corollary}
		\begin{proof}
    			Let $\tilde{t}>0$ be arbitrary and choose $n_{0}\ge 1$ such that $t_{n}>\tilde{t}+1$ for all $n\ge n_{0}$. For each $n\geq n_{0}$, define $v_n:[0,\tilde{t}+1]\to H^1(\mathbb{R}^N)$ by $v_n(t):=u_n(t+t_n-\tilde{t}-1)$ for $t\in [0,\tilde{t}+1]$. Then $v_{n}$ is a solution of~\eqref{A-Fn-eqn}. Choose $\gamma\in\mathbb{R}$ such that $k<\gamma<\inf\sigma_{ess}(A)-\lambda$ and $\lambda+\gamma\not\in\sigma(A)$.
            By Proposition \ref{prop-con-poinc}, we obtain the  estimate
\begin{align*}
\beta_{L^{2}}(\{u_{n}(t_{n}-1)\}_{n\ge 1}) & = \beta_{L^{2}} ( \{ u_n (t_{n}-1) \}_{n\ge n_{0}} ) = \beta_{L^{2}} ( \{v_n (\tilde t) \}_{n\ge n_{0}} ) \\
				&  \le e^{-(\gamma - k)\tilde t}\beta_{L^{2}}(\{v_{n}(0)\}_{n\ge n_{0}}) \leq  Re^{-(\gamma - k)\tilde t},
			\end{align*}
			where the last inequality follows from  $\|v_n(0)\|_{H^1}=\|u_n(t_n-\tilde t-1)\|_{H^1}\leq R$ for all $n\geq 1$. Since $\tilde t>0$ is arbitrary, we have
			$$
			\beta_{L^{2}}(\{u_{n}(t_n-1)\}_{n\ge 1})\leq R e^{-(\gamma - k)\tilde t} \to 0 \quad  \text{ as  } \ \ \tilde t \to +\infty,
			$$
			which shows that $\{ u_n (t_n-1) \}_{n\geq 1}$ is relatively compact in $L^2(\R^N)$.\\
			\indent Without loss of generality, we may assume that $u_n(t_n-1)=v_n(\tilde t)\to \bar v$ in $L^2(\R^N)$ for some $\bar v\in H^1(\R^N)$. Taking into account the bound $\| v_n (t) \|_{H^1}\leq R$ for all $t\in [\tilde t, \tilde t +1]$ and $n\geq 1$, 
and using the continuity properties of evolution equations (see \cite[Th.~3.2]{Cw-Luk-2021}), we conclude that $u_n(t_n) = v_n(\tilde t +1) \to v (\tilde t +1)$ in $H^1(\R^N)$, as $n\to +\infty$, where $v:[\tilde t,\tilde t+1]\to H^1(\R^N)$ is the solution of equation $\dot v(t) = - A v(t)+\lambda v(t)+F_0(v(t))$ with initial condition $v(\tilde t) = \bar v$.
		\end{proof}	
The following proposition provides continuity properties of a family of Nemytskii operators and establishes an estimate for their measure of noncompactness in $L^2(\R^N)$.
\begin{proposition}\label{prop-cond-f-n}
Assume the mappings $f_{n}:\R^N\times \R\to\R$, $n\ge 0$, satisfy condition $(f1)$ with a common function $c\in L^{2}(\R^{N}$ and condition $(f2)$ with the same Lipschitz constant $l\in L^{\infty}(\R^{N})$. Suppose that 
\begin{equation}\label{conv-fn-1}
f_{n}(x,u)\to f_{0}(x,u) \quad \text{ as } \ n\to \infty 
\end{equation}
for a.e. $x\in\R^{N}$ and $u\in\R$. If $F_n: H^1(\R^N)\to L^2(\R^N)$ is the Nemytskii operator determined by $f_{n}$, then, for any $u\in H^{1}(\R^{N})$, 
\begin{align}\label{conv-fn}
F_{n}(u)\to F_{0}(u) \quad \text{in } L^{2}(\mathbb{R}^{N}) \text{ as } n\to\infty.
\end{align}
Moreover, for any bounded set $W\subset H^1(\R^N)$, the set $\bigcup_{n\geq 1}F_{n}(W)$ is bounded in $L^2(\R^N)$ and 
\begin{align}\label{est-eq-non-2}   
\beta_{L^{2}}\left(\bigcup_{n\geq 1}F_n(W)\right) \le \hat \varrho (l) \beta_{L^{2}}(W).
\end{align} 
\end{proposition}
\begin{proof}
By assumption \eqref{conv-fn-1}, for every $u\in H^{1}(\R^{N})$ the convergence \eqref{conv-fn} holds pointwise in $\R^{N}$. Moreover, $f_n(x,0)=c(x)$ for a.e. $x\in\R^{N}$, which together with condition $(f2)$ implies that
\[
|[F_n(u)](x)|=|f_n(x,u(x))|\le c(x) + l(x)|u(x)|
\]
for a.e. $x\in\R^{N}$ and all $n\ge1$. Since $c + l|u|$ belongs to $L^{2}(\R^{N})$, the dominated convergence theorem implies \eqref{conv-fn}. Observe that for any $n\ge 1$, the mapping $F_{n}$ satisfies the Lipschitz condition 
\begin{align*}
\|F_{n}(u) - F_{n}(v)\|_{L^{2}} \le L\|u-v\|_{L^{2}}, \quad u,v\in H^{1}(\R^{N}),
\end{align*}
where $L:=\|l\|_{L^{\infty}}$. Let $M_W>0$ be such that $W\subset B_{H^{1}}(0,M_W) := \{u\in H^{1}(\R^{N}) \ | \ \|u\|_{H^{1}}< M_{W}\}$. Since $F_{n}(0) = c$, for any $u\in W$, we have
$$
\|F_{n}(u)\|_{L^2} \leq \|c\|_{L^{2}}+L\|u\|_{H^1} \leq \|c\|_{L^{2}}+L M_W, \quad n\ge 1,
$$
which shows that the set $\bigcup_{n\geq 1}F_{n}(W)$ is bounded in $L^{2}(\R^{N})$. To prove the inequality \eqref{est-eq-non-2}, we first show that for every $\ve >0$ there exists $n_{0} \ge 1$ such that 
\begin{equation}\label{ineq-a-1}
\beta_{L^{2}}\left(\bigcup_{n\geq n_{0}}F_n(W)\right) \le \tilde r_\ve := (\hat \varrho(l)+\ve)(\beta_{L^2}(W) +\ve) + \ve.
\end{equation}
Indeed, let us choose $R>0$ such that 
\begin{equation}\label{eq-epsilon-1}
	\hat\varrho(l) + \ve \ge l (x) \quad\text{for a.e. }\ \ |x|\ge R.
\end{equation}
Consider the decompositions $F_{n} = F_{n,1} + F_{n,2}$ for $n\ge 1$, where 
\begin{align*}
[F_{n,1} (u)](x) := \chi_{B(0,R)}(x)f_n (x,u(x)) \quad \text{and} \quad [F_{n,2}(u)](x) := (1-\chi_{B(0,R)}(x))f_n (x,u(x)) 
\end{align*}
for a.e. $x\in\R^N$ and all $u\in H^1(\R^N)$. We claim that $\bigcup_{n\geq 1}F_{n,1}(W)$ is a relatively compact subset of $L^{2}(\mathbb{R}^{N})$. Indeed, let $(g_k)$ be an arbitrary sequence such that $g_{k} = F_{n_{k},1}(u_{k})$ for some $n_{k}\ge 1$ and $u_{k}\in W$. Since the sequence $(u_k)$ is bounded in $H^1(\mathbb{R}^{N})$, the Banach--Alaoglu theorem implies that, up to a subsequence, $u_k \rightharpoonup u_0$ weakly in $H^1(\mathbb{R}^N)$. Now, consider the restrictions $\tilde u_k$ of $u_k$ to the ball $B(0,R)$. These restrictions form a bounded sequence in $H^1(B(0,R))$, which is compactly embedded in $L^2(B(0,R))$ by the Rellich--Kondrachov theorem.
Therefore, possibly passing to a further subsequence, we may assume that $(\tilde u_k)$ converges strongly in $L^2(B(0,R))$ to $\tilde u_0$, the restriction of $u_0$ to $B(0,R)$. Passing to a further subsequence if necessary, we may also assume that the sequence $(n_k)$ is either constant or $n_k \to \infty$ as $k \to \infty$. In the former case, we write $\tilde n:= n_{k}$ for $k\ge 1$ and 
\begin{align*}
\|g_{k} - F_{\tilde n, 1}(u_0) \|_{L^{2}(\R^{N})} = \|F_{\tilde n,1}(u_k) - F_{\tilde n,1}(u_0) \|_{L^{2}(\R^{N})} 
 \le \|l\|_{L^\infty} \| \tilde u_{k} - \tilde u_0 \|_{L^2(B(0,R))}.
\end{align*}
Hence $g_{k} \to F_{\tilde n,1}(u_0)$ as $k\to +\infty$ in $L^2(\R^N)$. In the latter case, where $n_k \to \infty$ as $k\to\infty$, we have
\begin{align*}
\|g_{k}-F_{0,1}(u_{0})\|_{L^2(\R^{N})} & = \|F_{n_{k},1} (u_{k})-F_{0,1}(u_{0})\|_{L^2(\R^{N})} \\
& \le \|F_{n_{k},1} (u_{k})-F_{n_{k},1}(u_{0})\|_{L^2(\R^{N})} + \|F_{n_{k},1} (u_{0})-F_{0,1}(u_{0})\|_{L^2(\R^{N})} \\
& \le \|l\|_{L^{\infty}}\|\tilde u_{k}- \tilde u_{0}\|_{L^2(B(0,R))} + \|F_{n_{k}} (u_{0})-F_{0}(u_{0})\|_{L^2(\R^{N})}.
\end{align*}
This, together with \eqref{conv-fn}, implies that $g_{k} \to F_{0,1}(u_{0})$ in $L^{2}(\R^{N})$ as $k\to \infty$. Consequently, the set $\bigcup_{n\geq 1}F_{n,1}(W)$ is relatively compact in $L^{2}(\R^{N})$. Let us now take a finite covering of $W$ by the balls $B_{L^2}(u_i, r_\ve)$ for $i=1,\ldots, k_\ve$, where 
$r_\ve := \beta_{L^2}(W) +\ve$. By \eqref{conv-fn}, we choose $n_{0} \ge 1$ such that
\begin{align*}
\|F_{n}(u_{i})-F_{0}(u_{i})\|_{L^2} \le \ve, \quad n\ge n_{0}, \ 1\le i \le k_{\ve}.
\end{align*}
We now check that the set $\bigcup_{n \ge n_{0}} F_{n,2}(W)$ can be covered by the balls $B_{L^2}(F_{0}(u_{i}), \tilde r_{\ve})$, $1\le i \le k_{\ve}$. Let $y = F_{n,2}(u)$ for some $n\ge n_{0}$ and $u\in W$. Then $u\in B_{L^2}(u_i, r_\ve)$ for some $1\le i \le k_\ve$. By condition $(f2)$ and inequality \eqref{eq-epsilon-1}, we have $$\|F_{n,2} (u)-F_{n,2}(u_{i})\|_{L^2} \leq (\hat \varrho(l)+\ve) \|u-u_{i}\|_{L^2}, \quad u \in H^1(\R^N),$$ which implies that 
\begin{align*}
\|y-F_{0,2}(u_{i})\|_{L^2} & = \|F_{n,2} (u)-F_{0,2}(u_{i})\|_{L^2} \\
& \le \|F_{n,2} (u)-F_{n,2}(u_{i})\|_{L^2} + \|F_{n,2} (u_{i})-F_{0,2}(u_{i})\|_{L^2} \\
& \le (\hat \varrho(l)+\ve)\|u-u_{i}\|_{L^2} + \|F_n (u_{i})-F_0(u_{i})\|_{L^2} \\
& \le (\hat \varrho(l)+\ve)r_{\ve} + \ve = \tilde r_{\ve}.
\end{align*}
Consequently, 
\begin{align*}
\beta_{L^{2}}\left(\bigcup_{n\geq n_{0}}F_n(W)\right) \le \beta_{L^{2}}\left(\bigcup_{n\geq 1}F_{n,1}(W)\right) + \beta_{L^{2}}\left(\bigcup_{n\geq n_{0}}F_{n,2}(W)\right) \le \tilde r_\ve
\end{align*}
and the estimate \eqref{ineq-a-1} follows. Observe that given $k\ge 1$, the estimate is also valid for the constant family $\{F_{k}\}_{n\ge 1}$, which implies that 
\begin{align*}
\beta_{L^{2}}(F_{k}(W)) \le \tilde r_{\ve}, \quad \ve > 0.
\end{align*}
Therefore,
\begin{align*}
\beta_{L^{2}}\left(\bigcup_{n\geq 1}F_n(W)\right) & = \max\left\{\beta_{L^{2}}(F_{1}(W)),\ldots, \beta_{L^{2}}(F_{n_{0}-1}(W)), \beta_{L^{2}}\left(\bigcup_{n\geq n_{0}}F_n(W)\right)\right\} \\
& \le \tilde r_{\ve} = (\hat \varrho(l)+\ve)(\beta_{L^2}(W) +\ve) + \ve,
\end{align*}
which proves the inequality \eqref{est-eq-non-2}, since $\ve > 0$ was arbitrary.
\end{proof}

		\section{Proofs of the main results}
We assume that for any $ u \in H^{1}(\mathbb{R}^{N})$, the nonlinear mapping $F$ is defined by
\begin{equation}\label{Nemytzki_f}
[F( u)](x) = f(x, u(x)) \quad \text{for a.e. } x \in \mathbb{R}^N,
\end{equation}
where $f$ satisfies the assumptions $(f1)$\,--\,$(f3)$. In particular, this ensures that $F$ is well defined and satisfies the Lipschitz condition \eqref{F-lip-cond} with constant $L = \|l\|_{L^\infty}$. Moreover, as a consequence of Proposition \ref{prop-cond-f-n}, we obtain the estimate
\begin{equation}\label{k-set-contr-F}   
\beta_{L^2} \left( F (W)\right)\leq \hat\varrho(l)\beta_{L^2} (W). 
\end{equation}
Equation \eqref{main-eq}, with the nonlinearity defined by \eqref{Nemytzki_f}, represents the abstract form of the parabolic equation $(P)_{\lambda}$. As noted at the beginning of Section~\ref{sec-3}, equation \eqref{main-eq} generates a continuous semiflow $\Phi$ on $H^1(\mathbb{R}^N)$. By Corollary \ref{cor-admissibility}, this semiflow is admissible with respect to any bounded subset of this space.

		\subsection{Nonresonance Conley index formula}
We will prove the existence of orbits connecting stationary points for the equation $(P)_\lambda$ in the case where the linearizations of the nonlinear perturbation at zero and at infinity are not in resonance with the spectrum of the operator $-\Delta + V - \lambda$. We will need the Lyapunov function property for the semiflow generated by equation \eqref{main-eq}, when the nonlinear perturbation is given by the Nemytskii operator \eqref{Nemytzki_f}.
\begin{proposition}\label{Lapunov-function-property}		
Let $f:\mathbb{R}^N \times \mathbb{R} \to \mathbb{R}$ be a mapping satisfying conditions $(f1)$, $(f2)$, and let 
$\mathcal{E} : H^1(\mathbb{R}^N) \to \mathbb{R}$ be the functional given by
\[
\mathcal{E}(v) := \frac{1}{2} \int_{\mathbb{R}^N} \big(|\nabla v(x)|^2+(V(x)-\lambda)v(x)^2\big)\, dx 
- \int_{\mathbb{R}^N} \mathcal{F}(x,v(x))\, dx, 
\quad v \in H^1(\mathbb{R}^N),
\]
where $\mathcal{F}: \mathbb{R}^N \times \mathbb{R} \to \mathbb{R}$ is defined by
\[
\mathcal{F}(x,v) := \int_0^{v} f(x,w)\,dw, \quad x\in\mathbb{R}^N, \ v\in\mathbb{R}.
\]
Then $\mathcal{E}$ is a well-defined continuous Lyapunov functional for equation \eqref{main-eq}. More precisely, if
$u:[t_0,t_1)\to H^1(\mathbb{R}^N)$ is a solution of \eqref{main-eq}, then
\[
\frac{d}{dt}\mathcal{E}(u(t)) = -\int_{\mathbb{R}^N} |\dot u(t)|^2\,dx, \quad t\in(t_0,t_1).
\]
\end{proposition}
The following theorem provides conditions under which 
$K_0 := \{0\}$ and the set $K_\infty$, consisting of all bounded solutions 
of the semiflow $\Phi$, are isolated invariant sets. 
It also gives explicit formulas for their Conley indices in terms of the 
multiplicities of the eigenvalues of the corresponding linearized operators.
\begin{theorem}\label{non-res-index-formulae}
Let $f:\mathbb{R}^N \times \mathbb{R}\to \mathbb{R}$ be a mapping satisfying conditions $(f1)$\,--\,$(f3)$, and suppose that 
$f(x,0)=0$ for a.e. $x\in\mathbb{R}^N$. \\[3pt]
\noindent\makebox[6.5mm][r]{$(i)$} \parbox[t][][t]{156mm}{If the linearization condition $(L_0)$ holds with 
$$\ker (-\Delta + V-\alpha)=\{ 0\} \quad\text{and}\quad\hat\varrho(|\alpha|) < \inf \sigma_{ess}(-\Delta+V)-\lambda,$$ then $K_0=\{ 0 \}$ is an isolated invariant set with respect to the semiflow $\Phi$ and
				$$
				h(\Phi, K_0) = \Sigma^{k_0},
				$$
				where $k_0 := d^- (V-\alpha,\lambda)$.}\\[3pt]
\noindent\makebox[6.5mm][r]{$(ii)$} \parbox[t][][t]{156mm}{If the linearization condition $(L_\infty)$ holds with $$\ker(-\Delta + V-\omega)=\{0\} \quad\text{and}\quad
				\hat \varrho(|\omega|) < \inf \sigma_{ess}(-\Delta+V)-\lambda,$$ then the set $K_\infty$, consisting of all bounded solutions of the semiflow $\Phi$, is an isolated invariant set with respect to the semiflow $\Phi$ and 
				$$
				h(\Phi, K_\infty) = \Sigma^{k_\infty},
				$$
				where $k_\infty := d^- (V-\omega,\lambda)$.}
\end{theorem}
		\begin{proof}
			(i) Let $h:\R^{N}\times\R\times[0,1] \to \R$ be a linear homotopy between $f$ and multiplication by $\alpha$, defined by
\begin{align*}
h(x,u,s) := (1-s)f(x,u) + s\alpha(x)u, \quad x \in \mathbb{R}^N, \ u\in\R, \text{ and } s \in [0,1],
\end{align*}
where $\alpha\in L^{\infty}(\R^{N})$. Note that for all $u_1, u_2\in\R$ and $s\in [0,1]$, we have
\begin{equation}
|h(x,u_1,s)-h(x, u_2, s) | \leq  l_{\alpha}(x) |u_1-u_2| \quad \text{ for a.e. } x\in \R^N, 
\end{equation}
where $l_{\alpha}(x):=\max\{ l (x), |\alpha(x)| \}$ for a.e. $x\in\R^N$. Let us consider the family of equations
\begin{equation}\label{eq-a-h}
\dot u(t)=-Au(t)+\lambda u(t)+H(u(t),s), \quad t >0,
\end{equation}
where the map $H:H^{1}(\R^{N})\times[0,1]\to L^{2}(\R^{N})$ is defined, for any $u\in H^{1}(\R^{N})$, by 
\[
[H(u,s)](x) := h(x,u(x),s) \quad \text{for a.e. } x \in \mathbb{R}^N \text{ and } s \in [0,1].
\]
Since $l_{\alpha} \in L^{\infty}(\mathbb{R}^{N})$, it follows from Proposition \ref{prop-cond-f-n} that $H$ is continuous and satisfies the Lipschitz condition
\begin{align*}
\|H(u,s) - H(v,s)\|_{L^{2}} \le L\|u-v\|_{H^{1}}, \quad u,v\in H^{1}(\R^{N}), \ s\in[0,1],
\end{align*}
where $L := \|l_{\alpha}\|_{L^{\infty}}$. Consequently, equation~\eqref{eq-a-h} generates a continuous family of semiflows $\{\Psi^{(s)}\}_{s\in[0,1]}$ (see \cite[Th.~3.4.1]{Henry}). By assumption~\eqref{cond-alpha}, we have
\begin{align*}
\hat\varrho(l_{\alpha}) = \max\{\hat\varrho(l), \hat\varrho(|\alpha|)\} 
< \inf \sigma_{ess}(-\Delta+V)-\lambda.
\end{align*}
This, together with Proposition \ref{prop-cond-f-n} and Corollary~\ref{cor-admissibility}, implies that any bounded subset of $H^1(\mathbb{R}^N)$ is admissible with respect to the family $\{\Psi^{(s)}\}_{s\in[0,1]}$.	

We claim that there exists $r_0>0$ such that $\overline{B(0,r_0)}$ is an isolating neighborhood for $K_0$ with respect to the semiflows $\Psi^{(s)}$ for all $s\in[0,1]$. Suppose, to the contrary, that this is not the case. Then there exist sequences $(s_n)$ in $[0,1]$ and $(u_n)$ of nonzero bounded full solutions of the semiflow $\Psi^{(s_n)}$ for $n \ge 1$, such that
\[
\rho_n := \sup_{t \in \mathbb{R}} \|u_n(t)\|_{H^1} \to 0 \quad \text{as } n \to \infty.
\]
			Without loss of generality, we may assume that
			$$ \|u_n(0)\|_{H^1} \geq (1-1/2n) \rho_n \ \mbox{ for all } \ n\geq 1.$$
			Clearly, the mapping $v_n := \rho_n^{-1} u_n$ is a full solution of the equation 
			\begin{equation}\label{n-th-evol-eq-zero}
				\dot v (t)= - Av(t) + \lambda v(t)+ F_n (v(t)), \quad t>0,
			\end{equation}
			where $F_n: H^1(\R^N)\to L^2(\R^N)$ is the Nemytskii operator determined by the function 
			$$
			f_n(x,v) := (1-s_n)\rho_n^{-1} f(x,\rho_n v) + s_n \alpha(x) v \ \text{ for  a.e. } x\in\R^N \text{ and all } \ v\in\R.
			$$
Then, for each $n \ge 1$, the function $f_n$ satisfies condition $(f1)$ with $c=0$ and condition $(f2)$ with the Lipschitz constant given by $\tilde l(x):=\max\{|\alpha(x)|,l(x)\}$ for $x\in\R^{N}$. Note also that by condition $(L_{0})$, the convergence holds
			$$
			f_n (x,v) \to \alpha (x)v \ \mbox{ as } \ n\to +\infty,
			$$
for every $v\in\R$  and a.e. $x\in\R^N$. Therefore, by Proposition \ref{prop-cond-f-n}, we have $F_n (u)\to m_\alpha (u)$ in $L^2(\R^N)$ for any $u\in H^1(\R^N)$, where $m_\alpha: H^1(\R^N)\to L^2(\R^N)$ is defined by $[m_\alpha(u)](x):=\alpha(x)u(x)$. Moreover, 
			$$
			\beta_{L^2} \left( \bigcup_{n\geq 1} F_n(W)\right)\leq \hat\varrho (\tilde l) \beta_{L^2}(W)
			$$
			for any bounded $W\subset H^1(\R^N)$, where 
			$$
			\hat\varrho (\tilde l) = \max\{\hat\varrho(l),\hat\varrho(|\alpha)| \} < \inf \sigma_{ess}(A)-\lambda.
			$$
Take any $\tau>0$ and define $w_n:\mathbb{R}\to H^1(\mathbb{R}^N)$ by $w_n(t):=v_n(t-n-\tau)$ for $ t\in\mathbb{R}$.
Then $w_n$ solves the equation \eqref{n-th-evol-eq-zero} and $v_n(-\tau)=w_n(n)$ for $n\ge 1$. Observe that the mappings $F_n$ satisfy the Lipschitz condition \eqref{F-lip-cond} with the same constant $L = \|\tilde l\|_{L^{\infty}}$. Hence, in view of the bound $\|w_n(t)\|_{H^1}\le 1$ for $t\in\mathbb{R}$ and Corollary \ref{cor-admissibility}, the sequence $(v_n(-\tau))$ contains a subsequence converging in $H^1(\mathbb{R}^N)$.
			Hence, for any $\tau>0$, there exists a subsequence $\left(v_{n_k}(-\tau)\right)$ that converges to some $\bar v_\tau \in H^1(\mathbb{R}^N)$. By the standard continuity of solutions \cite[Th.~3.4.1]{Henry}, the sequence of solutions $\left(v_{n_k}\right)$ converges uniformly on compact subsets of $[-\tau,+\infty)$ to a solution $v$ of the equation
\begin{equation}\label{limit-lin-eq-zero}
\dot v(t)= -Av(t)+\lambda v(t)+m_\alpha(v(t)), \quad t>0,
\end{equation}
with initial condition $v(-\tau)=\bar v_\tau$. 
Using a diagonal argument, we obtain a solution $v:\R\to H^1(\R^N)$ of \eqref{limit-lin-eq-zero} such that
			$\|v(t)\|_{H^1}\leq 1$ for all $t\in\R$ and $\|v(0)\|_{H^1}=1$. Since there exists a Lyapunov functional for \eqref{limit-lin-eq-zero} (see Proposition \ref{Lapunov-function-property}) and  $v$ is not identically zero, it follows that both the $\alpha$- and $\omega$-limit sets of $v$ contain distinct equilibria of \eqref{limit-lin-eq-zero}. This contradicts the assumption that $\ker (-\Delta + V-\alpha)=\{ 0\}$, thus proving the claim that there exists $r_0>0$ such that $\overline{B(0,r_0)}$ is an isolating neighborhood for $K_0$ with respect to the family of semiflows $\{\Psi^{(s)}\}_{s\in[0,1]}$. \\
\indent Hence, by the homotopy invariance of the Conley index and Theorem \ref{linear-conley-index} we obtain 
$$
h(\Phi, K_0) = h(\Psi^{(0)}, K_0) = h(\Psi^{(1)}, K_0)=\Sigma^{k_0},
$$
which completes the proof of (i). The proof of (ii) is analogous (compare \cite[Th.~5.1]{Cw-Kok}).
\end{proof}
\begin{proof}[Proof of Theorem \ref{14102025-1145}]
In view of Theorem \ref{non-res-index-formulae}, we have 
$$
h(\Phi, K_0) = \Sigma^{k_0}  \  \ \text{ and } \ \  h(\Phi, K_\infty) = \Sigma^{k_\infty}.
$$
In particular, both sets $K_{0}$ and $K_\infty$ are irreducible. Since $k_0\neq k_\infty$, we can apply Theorem \ref{rybakowski-irreducible} to obtain a full nonzero solution $u:\R\to H^1(\R^N)$ for the semiflow $\Phi$ such that either 
$\alpha(u)=K_0$ or $\omega(u)=K_0$. Then, by Proposition \ref{Lapunov-function-property}, the semiflow $\Phi$ is gradient-like, which implies that either the $\alpha$- or $\omega$-limit set of $u$ contains a nonzero stationary solution.
\end{proof}

\subsection{Resonance Conley index formula}	
We will prove the existence of orbits connecting stationary points for the equation $(P)_\lambda$ in the case where the linearization of the nonlinear perturbation at infinity is in resonance with the spectrum of the operator $-\Delta + V - \lambda$. Recall that in this case $X_0 := \ker(-\Delta + V - \lambda) \neq \{ 0 \}$ and that $f$ satisfies the condition \eqref{bound-f}. In view of condition $(f3)$, we can choose $\gamma\in\R$ such that
\begin{equation}\label{gamma-choice}
\hat \varrho(l) < \gamma < \inf\sigma_{ess}(-\Delta + V) - \lambda.
\end{equation}
Let us consider a family of evolution equations
		\begin{equation}\label{eq-diff-g}
			\dot u(t) = - A u(t) + \lambda u(t) + G(u(t),s),\quad t>0,
		\end{equation}
		where $G: H^1(\R^N) \times [0,1] \to L^2(\R^N)$ is given by 
		\begin{align*}
			G(u, s)  := P F (P u+s Q u) + s Q F(P u+s Q u)  \ \mbox{ for all } \ u\in H^1(\R^N), \ s\in [0,1],
		\end{align*}
		and $F:H^1(\R^N)\to L^2(\R^N)$ is the Nemytskii operator associated with $f$, as defined in \eqref{Nemytzki_f}.
		\begin{lemma}\label{F-satisfies-pseudo-lip-ineq}
Assume that conditions $(f1)$\,--\,$(f3)$ hold and that $f$ satisfies the inequality \eqref{bound-f} with $m\in L^{\infty}(\R^{N})$. Then, for any $\gamma\in\R$ such that 
$$
\hat \varrho (l) < \gamma < \inf \sigma_{ess} (-\Delta +V)-\lambda \quad\text{and}\quad \gamma+\lambda \not\in\sigma(-\Delta +V),
$$
the mapping $F$ satisfies conditions \eqref{pseudo-lip-ineq} and \eqref{pseudo-lip-vs-spectrum} with some positive constants $L_{\sigma}$ and $C_{\sigma}$.
\end{lemma}
\begin{proof}
Let us choose $L_\sigma > 0$ such that
$$
\hat \varrho(l) < L_\sigma \leq \gamma.
$$
In view of $(f3)$, there exists $r>0$ such that
$$
l(x) < L_\sigma \ \text{ for a.e. } \ x\in \R^N \setminus B(0,r). 
$$
Hence, using assumptions $(f1)$ and $(f2)$, we obtain that for all $u\in H^1(\R^N)$, 
$$
|f(x, u(x))| \leq L_\sigma | u(x)|+|f(x,0)| \ \mbox{ for a.e. } x\in \R^N \setminus B(0,r).
$$
Taking into account inequality~\eqref{bound-f}, we derive \eqref{pseudo-lip-ineq} with $C_\sigma :=\|f(\,\cdot\,, 0)\|_{L^2} + \|m\|_{L^\infty} |B(0,r)|^{1/2}$.
		\end{proof}
		\begin{lemma}\label{G-basic-props}
			The mapping $G$ is continuous and there exist constants $M_G>0$ and $L_G>0$ such that
			\begin{equation}\label{G-L-infty-bdd}
				\| G(u,s)\|_{L^\infty}\leq M_G, \quad u\in H^1(\R^N), \ s\in [0,1],
			\end{equation}
			\begin{equation}\label{G-Lipschitz}
				\|G(u_1, s)-G(u_2, s)\|_{L^2} \leq L_G \|u_1-u_2\|_{H^1}, \quad u_1, u_2\in H^1(\R^N), \ s\in [0,1].
			\end{equation}
			Moreover, there exist constants $0 < L_\sigma < \gamma$ and $C_\sigma > 0$ such that
			\begin{equation}\label{G-pseudo-lip}
				\|G(u,s)\|_{L^2} \leq L_\sigma ( \|Pu\|_{L^2} + \|Qu\|_{L^2} ) + C_\sigma, \quad u\in H^1(\R^N), \ s\in [0,1].
			\end{equation}
		\end{lemma}
		\begin{proof}
			First, observe that for all $u\in H^1(\R^N)$ and  $s\in [0,1]$, we have
			\begin{equation}\label{G-another-form}
				G(u,s)=(1-s)PF(Pu+sQu) + sF(Pu+sQu).
			\end{equation}
			Applying Proposition \ref{prop-inclusion} with $p = \infty$, we obtain
			$$
			\|G(u,s)\|_{L^\infty} \lesssim (1-s)\|F(Pu+sQu)\|_{L^\infty} + s\|F(Pu+sQu)\|_{L^{\infty}}\leq \|m\|_{L^\infty},
			$$
			which proves \eqref{G-L-infty-bdd}. To verify \eqref{G-Lipschitz}, take any $u_1, u_2\in H^1(\R^N)$ and $s\in [0,1]$. Then
			\begin{align*}
			\| G(u_1, s)-G(u_2,s)\|_{L^{2}} & \le \|F(Pu_1+sQu_1)-F(Pu_2+sQu_2)\|_{L^2} \\
			& \le \|l\|_{L^{\infty}}\|(Pu_1+sQu_1)-(Pu_2+sQu_2)\|_{L^{2}} \\
			& \lesssim\|u_1-u_2\|_{L^{2}} \le \|u_1-u_2\|_{H^1},
			\end{align*}
			where we used the fact that $F$ is Lipschitz with constant $L=\|l\|_{L^{\infty}}$. \\
			\indent Moreover, since \eqref{gamma-choice} holds, it follows from Lemma \ref{F-satisfies-pseudo-lip-ineq} that there exist constants $0 < L_\sigma < \gamma$ and $C_\sigma > 0$ such that inequality \eqref{pseudo-lip-ineq} is satisfied. Consequently,
			$$
			\|G(u,s)\|_{L^2}\leq \|F(Pu+sQu)\|_{L^2}\leq  L_\sigma ( \|Pu\|_{L^2}+s\|Qu\|_{L^2}) + C_\sigma,
			$$
			which proves \eqref{G-pseudo-lip} and completes the proof.
		\end{proof}
\begin{lemma}\label{prop-cond-g}
			The following inequality holds
			\begin{align}\label{est-eq-non-2g}
				\beta_{L^{2}}(G(W\times[0,1])) \le \hat \varrho (l) \beta_{L^{2}}(W),
			\end{align}
			for any bounded set $W \subset H^{1}(\R^{N})$.
\end{lemma}
\begin{proof} 
Using \eqref{G-another-form} and the fact that $X_0$ is finite-dimensional, we obtain
			\begin{align*}
				\beta_{L^{2}}(G(W \times [0,1])) & \le \beta_{L^{2}}(\{s F (P u+s Q u) \ |\ u\in W, \ s\in [0,1]\}) \\
				& \le \beta_{L^{2}}(\mathrm{conv}(\{0\}\cup \{F (P u+s Q u) \ |  \ u\in W, \ s\in [0,1] \}) \\
				& \le \beta_{L^{2}}(\{F (P u+s Q u) \ | \ u\in W, \ s\in [0,1]\}).
			\end{align*}
			If we denote $W_{0} := \{(1-s)P u \ | \ u\in W, \ s\in [0,1] \}$, then by the inequality \eqref{k-set-contr-F}, we have
			$$
			\begin{aligned}
				\beta_{L^{2}}(G(W\times [0,1])) & \le \hat \varrho (l) \beta_{L^{2}}(\{(1-s)P u+ s u \ | \ u\in W, \ s\in[0,1] \}) \\
				& \le \hat \varrho(l)\beta_{L^{2}}(W_{0}) + \hat \varrho (l) \beta_{L^{2}}(\mathrm{conv}(\{0\} \cup W)).
			\end{aligned}
$$
Since $W_{0}$ is bounded and contained in a finite-dimensional space, we have $\beta_{L^{2}}(W_{0}) = 0$. Together with the equality $\beta_{L^{2}}(\mathrm{conv}(\{0\} \cup W))=\beta_{L^2}(W)$, this implies estimate \eqref{est-eq-non-2g}, and the proof is complete.
\end{proof}
We follow \cite[Th.~4.3]{MR3072663} (or \cite[Prop.~5.1]{MR4087267}) and \cite[Lem.~5.2]{Cw-Kr-2019} to show that the Landesman--Lazer conditions have geometric consequences in the phase space of the corresponding parabolic semiflows. In our setting, the nonlinearity $f$ is bounded by a function $m \in L^\infty(\mathbb{R}^N)$, which, in particular, allows us to consider sets $W$ that are not necessarily bounded in $H^{1}(\R^{N})$.
\begin{theorem}\label{lem-est2}
Let $F$ be the mapping given by \eqref{Nemytzki_f}, and let $W\subset H^1(\R^N)$ be bounded in $L^\infty(\R^N)$. 
\begin{enumerate}
\item[$(i)$] If the condition $(LL)_+$ holds, then there exist $R > 0$ and $\nu>0$ such that 
$$
\langle F(v + w), v\rangle_{L^2} > \nu
$$
for all $w\in W$ and $v\in X_0$ with $\|v\|_{L^2} \ge R$.
\item[$(ii)$] If the condition $(LL)_-$ holds, then there exist $R > 0$ and $\nu>0$ such that
$$
\langle F(v + w), v\rangle_{L^2} < -\nu
$$
for all $w\in W$ and $v\in X_0$ with $\|v\|_{L^2} \ge R$.
\end{enumerate}
\end{theorem}
\begin{proof} We shall prove part $(i)$; part~$(ii)$ can be proved analogously.
Suppose, to the contrary, that there are sequences $(w_n)$ in $W$ and $(v_n)$ in $X_0$ such that $\|v_n\|_{L^2} \to \infty$ as $n\to +\infty$ and
\begin{equation}\label{g1b}
\<F(v_n + w_n), v_n\>_{L^2} \le 1/n \quad\text{ for all } \ \ n\ge 1.
\end{equation}
Recall that $X_0$ is finite-dimensional and satisfies $X_0 \subset L^1(\mathbb{R}^N)$ -- see \eqref{est-eigen}. For $n \ge 1$, define $z_n := v_n / \|v_n\|_{L^1}$. Passing to a subsequence if necessary, we may assume that $z_n \to \varphi$ in $L^1(\mathbb{R}^N)$ for some $\varphi \in X_0$, and that $z_n(x) \to \varphi(x)$ for a.e. $x \in \mathbb{R}^N$. We may then rewrite \eqref{g1b} as
			\begin{equation}\label{g7b}
				\<F(v_n + w_n), z_n - \varphi \>_{L^2} + \<F( v_n + w_n), \varphi \>_{L^2} \le ( n \| v_n \|_{L^2} )^{-1}, \quad n\ge 1.
			\end{equation}
			By condition \eqref{bound-f}, we have
			\begin{equation}\label{g6b}   
				\left| \<F(v_n + w_n), z_n - \varphi \>_{L^2} \right|
				\le \|m\|_{L^\infty} \| z_n - \varphi \|_{L^{1}}\to 0,\quad n\to\infty.
			\end{equation}
			Define $\Omega_+ := \{ \varphi > 0\}$ and $\Omega_- := \{\varphi < 0\}$. Then
			\begin{equation}
				\begin{aligned}\label{g5b} 
					&\<F(v_n + w_n), \varphi\>_{L^2} = \int_{\R^{N}} f(x, v_n(x) + w_n(x)) \varphi (x) \,d x \\ 
					& \quad = \int_{\Omega_+} f(x, v_n(x) + w_n(x)) \varphi (x) \,d x  + \int_{\Omega_-} f(x, v_n(x) + w_n(x)) \varphi (x) \,d x.
				\end{aligned}
			\end{equation}
			By assumption, there exists $r>0$ such that $\|w\|_{L^{\infty}}\le r$ for all $w \in W$, which yields
\begin{align*}
v_n(x) + w_n(x) =   \|v_n\|_{L^2} z_n(x) + w_n(x) \ge   \|v_n\|_{L^2}  z_n(x) - r
\end{align*}
and hence
\begin{equation*}
				v_n(x) + w_n(x) \to +\infty  \quad \text{ as } n\to \infty,  \ \ \text{ for a.e. } \ x\in \Omega_{+}.
\end{equation*}
			Therefore, by Fatou's lemma,
			\begin{equation}\label{g9b}
				\liminf_{n\to\infty}\int_{\Omega_+} f(x, v_n(x) + w_n(x)) \varphi (x) \,d x \geq
				\int_{\Omega_+} \check f_+(x) \varphi (x) \,d x.
			\end{equation}
			Similarly,
\begin{align*}
v_n(x) + w_n(x) =   \|v_n\|_{L^2} z_n(x) + w_n(x) \le   \|v_n\|_{L^2}  z_n(x) + r,
\end{align*}	
which implies  		
\begin{equation*}
				v_n(x) + w_n(x) \to -\infty  \quad \text{ as } n\to \infty,  \ \ \text{ for a.e. } \ x\in \Omega_{-}.
\end{equation*}			
Hence, an application of Fatou's lemma yields 
			\begin{equation}\label{g4b}
				\liminf_{n\to\infty}\int_{\Omega_-} f(x, v_n(x) + w_n(x)) \varphi (x) \,d x \geq
				\int_{\Omega_-} \hat f_-(x) \varphi (x) \,d x.
			\end{equation}
Combining \eqref{g7b}, \eqref{g6b}, \eqref{g5b}, \eqref{g9b}, and \eqref{g4b} yields
			\begin{equation*}
				\int_{\Omega_+} \hat f_+(x) \varphi (x) \,d x + \int_{\Omega_-} \hat f_-(x) \varphi (x) \,d x \leq 0,
			\end{equation*}
			which contradicts condition~$(LL)_+$, since $\|\varphi\|_{L^1}=1$.
\end{proof} 
		
\begin{remark}
If $V$ is a Kato--Rellich potential, then, in view of the unique continuation 
property (see~\cite{Gossez}), the nodal set $\{x\in\mathbb{R}^N\mid\varphi(x)=0\}$ 
has measure zero for any $\varphi\in\ker(A-\lambda)\setminus\{0\}$. Hence, the 
Landesman--Lazer condition~$\widetilde{(LL)}_{\pm}$ implies~$(LL)_{\pm}$.
\end{remark}
We now use the above geometric properties to determine the Conley index of the maximal invariant set of the semiflow $\Phi$.
\begin{theorem}\label{18072019-0920}
Assume that one of the conditions $(LL)_{\pm}$ holds. Then the maximal invariant set $K_\infty$ for the semiflow $\Phi$ determined by $(P)_\lambda$ is isolated in $H^{1}(\R^{N})$, and its Conley index is given by 
\begin{equation}\label{resonant-Conley-index-formula}
h(\Phi, K_\infty)=\left\{\begin{aligned}
&\Sigma^{d_\infty^+} && \text{ if } \ (LL)_+ \ \text{holds},\\ 
&\Sigma^{d_\infty} && \text{ if } \ (LL)_-  \ \text{holds},
\end{aligned}\right.
\end{equation}
where $d_\infty := d(V,\lambda)$ and $d_\infty^+:=d(V,\lambda) + \dim \ker (-\Delta+V-\lambda)$.
\end{theorem}
\begin{proof} 
For each $s \in [0,1]$, let $\Psi^{(s)}$ denote the semiflow on $H^1(\mathbb{R}^N)$ generated by equation~\eqref{eq-diff-g}. 
In view of Lemma~\ref{G-basic-props} and the general continuity properties of the solutions (see \cite[Th.~3.4.1]{Henry}), we infer that the family $\{\Psi^{(s)}\}_{s \in [0,1]}$ is continuous.
Moreover, Lemma~\ref{prop-cond-g} and Corollary~\ref{cor-admissibility} imply that every bounded subset of $H^1(\mathbb{R}^N)$ is admissible with respect to $\{\Psi^{(s)}\}_{s \in [0,1]}$.\\    
\indent Let $u$ be a full solution of the semiflow $\Psi^{(s)}$ for some $s\in[0,1]$ that is bounded in $H^1(\mathbb{R}^N)$. Applying the operator $P$ to the Duhamel formula \eqref{duhamel_eq}, we obtain 
\begin{equation*}
	Pu(t) = S_{A-\lambda}(t-t')Pu(t')+ \int_{t'}^{t} S_{A-\lambda}(t-\tau)PG(u(\tau),s) \,d\tau,\quad  t > t'.
\end{equation*}
Using the inclusion $\ker(A-\lambda) \subset \ker(I - S_{A-\lambda}(t))$, we infer that
$$
Pu(t) = Pu(t')+ \int_{t'}^{t} PF(Pu(\tau)+sQu(\tau)) \,d\tau, \quad t>t',
$$
and consequently
$$
\frac{d}{dt}Pu(t) = PF(Pu(t)+sQu(t))  \ \text{ for all } \ t\in\R.
$$			
Hence, for all $t \in \mathbb{R}$,
\begin{align}\label{eq-1-d4} 
	\frac{1}{2}\frac{d}{dt}\|Pu(t)\|^{2}_{L^{2}} = \left\langle \frac{d}{dt}Pu(t), Pu(t) \right\rangle_{L^{2}} = \langle PF(Pu(t)+sQu(t)), Pu(t) \rangle_{L^{2}}.
\end{align} 
Now observe that in view of Propositions \ref{prop-est-infty-1} and \ref{prop-est-infty-2}, as well as the estimate \eqref{G-L-infty-bdd}, there exists a constant $R_{Q} > 0$ such that 
\begin{align}\label{Q-infty-estimate}
\|Qu(t)\|_{L^{\infty}} \le \|Q_{+}u(t)\|_{L^{\infty}} + \|Q_{-}u(t)\|_{L^{\infty}} \le R_{Q} \ \text{ for all } \ t\in\R.
\end{align}
It then follows from Theorem  \ref{lem-est2} that there exist $\nu>0$ and $R_0>0$ such that
\begin{align}\label{gen-app-of-geom-cond}
\pm \langle F(v+w), v\rangle_{L^2}>\nu \ \text{ for } v\in X_0, \ w\in H^{1}(\R^{N}), \ \text{with } \|v\|_{L^{2}}\geq R_0 \ \text{ and } \|w\|_{L^\infty}\leq R_Q 
\end{align}
provided that condition $(LL)_{\pm}$ holds.

We claim that for any full bounded solution $u:\R\to H^1(\R^N)$ of the semiflow $\Psi^{(s)}$ with $s\in [0,1]$, 
\begin{equation}\label{P-bound-N-set}
\|Pu(t)\|_{L^2} < R_0 \ \mbox{ for all } \ t\in\R.
\end{equation}
Indeed, suppose that for some full bounded solution $u$ and $t_0\in\R$ we have $\|Pu(t_0)\|_{L^2}\geq R_0$.\\
\indent If condition $(LL)_+$ holds, define
$$
I_{+} := \{ t\in [t_0, +\infty) \mid \|P u(\tau)\|_{L^2} \geq R_0 \ \mbox{ for all }\  \tau \in [t_0,t]\}.
$$
Note that $\sup I_{+} = +\infty$. Otherwise, letting $t_{+} := \sup I_{+} < +\infty$, we clearly have $t_{+} \in I_{+}$. Combining \eqref{eq-1-d4}, \eqref{Q-infty-estimate}, and \eqref{gen-app-of-geom-cond}, we obtain
\begin{align*}
\left.\frac{1}{2}\frac{d}{dt}\|Pu(t)\|^{2}_{L^{2}} \right|_{t=t_{+}} & =  \langle PF(Pu(t_{+})+sQu(t_{+})), Pu(t_{+}) \rangle_{L^{2}} \\
& =\langle F(Pu(t_{+})+sQu(t_{+})), Pu(t_{+}) \rangle_{L^{2}} > \alpha,
\end{align*}
which contradicts the definition of $t_{+}$. Hence,
$$
\| Pu(t)\|_{L^2} \geq R_0 \quad \text{ for all } \ t\in [t_0, +\infty).
$$
Using \eqref{eq-1-d4} again, it follows that
\begin{align*}
\frac{1}{2}\frac{d}{dt}\|Pu(t)\|^{2}_{L^{2}} = \langle F(Pu( t)+sQu( t)), Pu( t) \rangle_{L^{2}} > \nu, \quad t \ge t_{0}
\end{align*}
and consequently $u$ cannot be bounded, a contradiction. This proves \eqref{P-bound-N-set} in the case of condition $(LL)_+$. Similarly, if $(LL)_-$ holds, define
$$
I_{-} := \{ t\in (-\infty, t_0] \mid \|P u(\tau)\|_{L^2} \geq R_0 \ \mbox{ for all }\  \tau \in [t,t_0]\}
$$
and observe that $t_{-}:=\inf I_{-}=-\infty$. Otherwise, letting $t_{-}:=\inf I_{-}>-\infty$, by analogous arguments we would have
$$
\left.\frac{1}{2}\frac{d}{dt}\|Pu(t)\|^{2}_{L^{2}} \right|_{t=t_{-}} =  \langle F(Pu(t_{-})+sQu(t_{-})), Pu(t_{-}) \rangle_{L^{2}} <- \nu,
$$
which contradicts the definition of $t_{-}$ and implies that
$$
\frac{1}{2}\frac{d}{dt}\|Pu(t)\|^{2}_{L^{2}} < -\nu \ \text{ for all } \ t\leq t_0.
$$
Hence, $u$ cannot be bounded, yielding a contradiction that proves \eqref{P-bound-N-set} in the case of $(LL)_-$. \\
\indent Now, Proposition \ref{Q-boundedness} together with \eqref{P-bound-N-set} yields $R_\infty > 0$, depending on $R_0$, such that for any bounded solution $u:\mathbb{R} \to H^1(\mathbb{R}^N)$ of $\Psi^{(s)}$ with $s \in [0,1]$, we have
$$ 
\|Q_{+}^{e}u(t)\|_{H^1} < R_\infty \quad \text{for all} \ \ t\in\R.
$$ 
Since the spaces $X_{-}$ and $X_{+}^{d}$ are finite-dimensional, by Proposition \ref{prop-est-infty-2}, we may increase $R_\infty$ if necessary to obtain
$$\|Qu(t)\|_{H^1} < R_\infty \quad \text{for all} \ \ t\in\R.$$
Together with \eqref{P-bound-N-set}, this shows that the set
\begin{equation}\label{isol-nbhd-res-case}
M := \{ u\in H^1(\R^N) \mid \|Pu\|_{L^2} \leq R_0 + 1\ \mbox{ and } \ \|Qu\|_{H^1} \leq R_\infty + 1\} 
\end{equation}
is an isolating neighborhood of the maximal invariant set $K_\infty^{s} := \mathrm{Inv}\,(M,\Psi^{(s)})$ with respect to the semiflow $\Psi^{(s)}$ for any $s\in [0,1]$.\\
\indent Since $M$ is admissible with respect to the family $\{\Psi^{(s)}\}_{s\in[0,1]}$, we can apply the homotopy invariance of the Conley index to obtain
$$
h(\Phi, K_\infty) = h(\Psi^{(1)}, K_{\infty}^{1}) = h(\Psi^{(0)},K_{\infty}^{0}).
$$
Then $\Psi^{(0)}$ is conjugate to the product semiflow $\Phi_P \times \Phi_Q$, where $\Phi_P : X_0 \to X_0$ is generated by $$\dot v(t) = PF(v(t)), \quad t>0$$ and
$\Phi_Q : \mathrm{Im}\, Q \to \mathrm{Im}\, Q$ corresponds to the linear equation
\[
\dot w(t) = -Aw(t) + \lambda w(t), \quad t > 0.
\]
By Theorem \ref{linear-conley-index}, we have $\mathrm{Inv}\,(M_{Q}, \Phi_{Q}) = \{0\}$, where $M_{Q} := \{ u \in X_-\oplus X_{+} \mid \|u\|_{H^{1}}\leq R_\infty\}$, and 
$$
h(\Phi_Q, \{ 0\})= \Sigma^{d_\infty}.
$$
Moreover, by the multiplicative property of the Conley index,
$$
h(\Psi^{(0)},K_{0}) = h(\Phi_P \times \Phi_Q, K_P \times \{0\}) = h(\Phi_P, K_P )\wedge h( \Phi_Q,\{0\}),
$$
where $K_P := \mathrm{Inv}(M_{P}, \Phi_P)$ with $M_{P} := \{ u \in X_0 \mid \|u\|_{L^2}\leq R_0\}$. Using \eqref{gen-app-of-geom-cond}, we see that 
\begin{align*}
\pm \langle F(v), v\rangle_{L^2}>\alpha \quad \text{ for } \ v\in X_0 \text{ with } \|v\|_{L^{2}} = R_0
\end{align*}
provided that condition $(LL)_{\pm}$ is satisfied. This implies that $M_{P}$ is an isolating block for semiflow $\Phi_{P}$ and
\begin{equation*}
h(\Phi_P, K_P) = \left\{\begin{aligned} &\Sigma^{\dim X_0} && \text{ if } (LL)_+ \text{ holds};\\
			    &\Sigma^{0} && \text{ if } (LL)_-  \text{ holds}.
            	\end{aligned}\right.
\end{equation*}
Since
$$
\Sigma^{d_\infty} \wedge \Sigma^0 = \Sigma^{d_\infty} \quad \text{ and } \ \ \Sigma^{d_\infty}\wedge \Sigma^{\dim X_0} = \Sigma^{d_\infty^+},
$$
the above equalities yield index formula \eqref{resonant-Conley-index-formula}.
\end{proof}
		
\begin{proof}[Proof of Theorem \ref{30042019-1204}]
First, observe that by Theorem \ref{non-res-index-formulae}\,(i) the set $K_0=\{0\}$ is isolated and invariant with respect to the semiflow $\Phi$, and 
$$
h(\Phi, K_0) = \Sigma^{d(V-\alpha, \lambda)}.
$$
On the other hand, by Theorem \ref{18072019-0920}, the set $K_\infty$ is bounded, and its Conley index is given by formula \eqref{resonant-Conley-index-formula}. This implies, in particular, that $K_\infty$ is irreducible and $h(\Phi, K_\infty)\neq \bar 0$. Moreover, assumptions~(i) and~(ii) ensure that
$h(\Phi, K_\infty)\neq h(\Phi, K_0)$. Hence, by Theorem \ref{rybakowski-irreducible}, there exists a nonzero solution $u:\R\to H^1(\R^N)$ of $\Phi$ such that either $\alpha(u)=K_0$ or $\omega(u)=K_0$. Since the semiflow $\Phi$ is gradient-like by Proposition \ref{Lapunov-function-property}, we deduce that either $\omega(u)$ or $\alpha(u)$ must contain a nonzero stationary solution. This completes the proof of the theorem.
\end{proof}

\begin{appendices}

\section{Appendix}

\noindent {\bf Sectorial Operators in Banach Spaces}. Let $A: D(A) \subset X \to X$ be a closed operator on a real Banach space $X$. The spectrum $\sigma(A)$ and the resolvent set $\rho(A)$ are defined via its complexification $A_{\mathbb{C}}$ acting on the complexified space $X_{\mathbb{C}} = X \oplus iX$, where $A_{\mathbb{C}}(u + iv) = Au + iAv$ for $u, v \in D(A)$. Specifically, we set $\sigma(A) := \sigma(A_{\mathbb{C}})$ and $\rho(A) := \rho(A_{\mathbb{C}})$.

Assume that $\sigma_{-} := \{\lambda_1, \dots, \lambda_n\} \subset \mathbb{R}$ is a finite set of real isolated eigenvalues of $A$, separated from the remainder of the spectrum, $\sigma_{+} := \sigma(A) \setminus \sigma_{-}$. For each $\lambda_j$, we define the Riesz projection $P_{j,\mathbb{C}}$ on $X_{\mathbb{C}}$ as
\begin{equation}\label{riesz-proj}
    P_{j,\mathbb{C}} := \frac{1}{2\pi i} \int_{\gamma_{j}} (\mu I - A_{\mathbb{C}})^{-1} d\mu,
\end{equation}
where $\gamma_j \subset \rho(A_{\mathbb{C}})$ is a small circle enclosing only $\lambda_j$. Since these contours can be chosen to be symmetric with respect to the real axis, the projections $P_{j,\mathbb{C}}$ restrict to real operators $P_j$ on $X$. Following the spectral theorem (see \cite[Th. 1.5.2]{Henry}), the total projection $P_{-} = \sum_{j=1}^n P_j$ induces a direct sum decomposition $X = X_{-} \oplus X_{+}$, where $X_{+} := \ker P_{-}$ and 
\begin{equation}
    X_{-} = \bigoplus_{j=1}^n X_j, \quad \text{with} \quad X_j = P_{j}(X).
\end{equation}
Both subspaces $X_{-}$ and $X_{+}$ are invariant under $A$, i.e., $A(X_{-}) \subset X_{-}$ and $A(X_{+} \cap D(A)) \subset X_{+}$. If we denote the restrictions of $A$ to these subspaces by $A_{-} = A|X_{-}$ and $A_{+} = A|X_{+} \cap D(A)$, it follows that their spectra satisfy $\sigma(A_{-}) = \sigma_{-}$ and $\sigma(A_{+}) = \sigma_{+}$. 

Next, we recall the definition of a broad class of operators used in this work. A linear, closed and densely defined operator $A: D(A) \subset X \to X$ is said to be \emph{sectorial} if there exist $\phi \in (0, \pi/2)$, $M \ge 1$, and $a \in \mathbb{R}$ such that the sector 
\begin{equation*}
    S_{a,\phi} := \{ \lambda \in \mathbb{C} : \phi \le |\arg(\lambda - a)| \le \pi, \, \lambda \neq a \}
\end{equation*}
is contained in the resolvent set $\rho(A)$, and the following estimate holds:
\begin{equation*}
    \|(\lambda I - A)^{-1}\| \le \frac{M}{|\lambda - a|} \quad \text{for all } \lambda \in S_{a,\phi}.
\end{equation*}
If $A$ is sectorial, then $-A$ is the infinitesimal generator of an analytic semigroup $\{S_A(t)\}_{t \ge 0}$. Furthermore, $A$ is called \emph{positive} if $\mathrm{Re}\,\mu > 0$ for all $\mu \in \sigma(A)$. For a positive sectorial operator, we define the fractional powers $A^{-\alpha}$ for $\alpha > 0$ as
\begin{equation*}
    A^{-\alpha} := \frac{1}{\Gamma(\alpha)} \int_0^\infty t^{\alpha - 1} S_A(t) \, dt.
\end{equation*}
The associated \emph{fractional power space} is $X^\alpha := D(A^\alpha)$ equipped with the graph norm $\|x\|_\alpha := \|A^\alpha x\|$. For further details, we refer the reader to \cite{Henry}, \cite{Pazy}.\\

\noindent {\bf Conley index}. Finally, we briefly recall a version of the Conley index due to Rybakowski (see \cite{rybakowski} or \cite{rybakowski-TAMS}). Let $\Phi \colon [0,+\infty) \times X \to X$ be a semiflow on the space $X$. A continuous map $u \colon J \to X$, where $J \subset \mathbb{R}$ is an interval, is called a \emph{solution of $\Phi$} if
\[
u(t+s) = \Phi(t,u(s))
\]
for all $t \geq 0$ and $s \in J$ such that $t+s \in J$. In particular, if $J=\mathbb{R}$, then $u$ is called a \emph{full solution} of the semiflow $\Phi$. In this case, the $\alpha$- and \emph{$\omega$-limit sets} of $u$ are defined by
\[
\alpha(u) := \left\{ x = \lim_{n\to\infty} u(t_n) \,\middle|\, t_n \to -\infty \right\} \quad\text{and}\quad \omega(u) := \left\{ x = \lim_{n\to\infty} u(t_n) \,\middle|\, t_n \to +\infty \right\},
\]
respectively. For $N \subset X$, we define $\mathrm{Inv}(N,\Phi)$ as the set of all points $x \in N$ for which there exists a full solution $u \colon \mathbb{R} \to X$ with $u(0) = x$ and $u(\mathbb{R}) \subset N$.
A set $K \subset X$ is called \emph{invariant} with respect to the semiflow $\Phi$ if $\mathrm{Inv}(K, \Phi) = K$. The set $K$ is an \emph{isolated invariant set} if there exists a closed set $N \subset X$ such that $K = \mathrm{Inv}(N, \Phi) \subset \mathrm{int}\,N$. In this case, $N$ is called an \emph{isolating neighborhood} of $K$. We say that a set $N \subset X$ is \emph{admissible with respect to $\Phi$} if, for any sequence $(t_n)$ in $[0,+\infty)$ with $t_n \to +\infty$ and any sequence $(x_n)$ in $X$ such that
\[
\{\Phi(t, x_n) \mid t \in [0, t_n]\} \subset N \quad \text{for all } n \ge 1,
\] 
the set $\{\Phi(t_n, x_n) \ | \ n\ge 1\}$ is relatively compact in $X$. 

We say that the family of semiflows $\{\Phi^{(s)}\}_{s \in [0,1]}$ on $X$ is \emph{continuous} if the map
$(t,x,s) \mapsto \Phi^{(s)}(t,x)$
is continuous on $[0,+\infty) \times X \times [0,1]$. Moreover, a set $N \subset X$ is \emph{admissible} with respect to this family if, for any sequences $(t_n)$ in $[0,+\infty)$ with $t_n \to +\infty$, $(x_n)$ in $X$, and $(s_n)$ in $[0,1]$ such that 
\[
\left\{\Phi^{(s_n)}(t,x_n) \mid t \in [0,t_n]\right\} \subset N \quad \text{for all } n \ge 1,
\] 
the set $\left\{\Phi^{(s_n)}(t_n,x_n) \ | \ n \ge 1\right\}$ is relatively compact in $X$.
		
		Let ${\mathcal I}(X)$ denote the family of all pairs $(\Phi, K)$, where $\Phi$ is a semiflow on $X$ and $K \subset X$ is an isolated invariant set for $\Phi$ that admits an admissible isolating neighborhood.
        If $(\Phi, K) \in {\mathcal I}(X)$, then the \emph{Conley homotopy index} $h(\Phi,K)$ of $K$ relative to $\Phi$ is defined by
\[
h(\Phi, K) := [(B/B^-, [B^-])],
\]
where $B$ is an isolating block of $K$ (see \cite{rybakowski}) with exit set $B^- \neq \emptyset$. If $B^- = \emptyset$, we set
\[
h(\Phi, K) := [(B \cup \{a\}, a)],
\]
where $a$ is an arbitrary point outside $B$. In particular, $h(\Phi, \emptyset) = \overline{0}$, where $\overline{0} := [(\{a\}, a)]$.

We now list several properties of the homotopy index. \\[3pt]
\noindent\makebox[9.5mm][l]{(H1)}\parbox[t][][t]{156mm}{If $(\Phi, K)\in {\mathcal I}(X)$ is such that $h(\Phi, K)\neq \overline{0}$, then $K\neq \emptyset$.}\\[5pt]
\noindent\makebox[9.5mm][l]{(H2)}\parbox[t][][t]{156mm}{If $(\Phi, K_1), (\Phi, K_2)\in {\mathcal I}(X)$ and $K_1\cap K_2=\emptyset$, then $(\Phi, K_1\cup K_2)\in {\mathcal I}(X)$ and $$h(\Phi, K_1\cup K_2) = h(\Phi, K_1)\vee h(\Phi, K_2).$$}\\
\noindent\makebox[9.5mm][l]{(H3)}\parbox[t][][t]{156mm}{If $(\Phi_1, K_1)\in {\mathcal I}(X_1)$ and $(\Phi_2, K_2)\in {\mathcal I}(X_2)$, then $(\Phi_1\times \Phi_2, K_1\times K_2)\in {\mathcal I} (X_1\times X_2)$ and $$h(\Phi_1\times \Phi_2, K_1\times K_2) = h(\Phi_1, K_1) \wedge h(\Phi_2, K_2).$$}\\
\noindent\makebox[9.5mm][l]{(H4)}\parbox[t][][t]{156mm}{Let $N \subset X$ be a closed set that is admissible with respect to the continuous family of semiflows $\{\Phi^{(s)}\}_{s \in [0,1]}$ and suppose that $K_s := \mathrm{Inv}(N,\Phi^{(s)}) \subset \mathrm{int}\,N$ for all $s \in [0,1]$. Then
\[
h(\Phi^{(0)}, K_0) = h(\Phi^{(1)}, K_1).
\]}\\
		\noindent  In the case of a linear semiflow, the Conley index can be computed using the following formula.
\begin{theorem} \label{linear-conley-index}{\em (See \cite[Ch. I, Th. 11.1]{rybakowski})}
Assume that a $C_0$ semigroup $\{T(t)\}_{t\geq 0}$ of bounded linear operators on a Banach space $X$ is hyperbolic (see, e.g., \cite[Def. V.1.14]{Engel}). If the dimension $\dim X_u=k$ of the unstable subspace $X_u$ \footnote{\ The unstable space $X_u$ is equal to $\ker P$, where $P$ is the spectral projection corresponding to $\{\lambda\in\sigma(T(t_0))\mid |\lambda|<1\}$ for some $t_0>0$.} is finite, then $\Phi\colon [0,+\infty) \times X\to X$, given by $\Phi(t,x):=T(t)x$, is a semiflow on $X$, $\{ 0\}$ is the maximal bounded invariant set with respect to $\Phi$, $(\Phi, \{ 0 \})\in {\mathcal I}(X)$ and $h(\Phi, \{ 0\})=\Sigma^k$ where $\Sigma^k=[(S^k, \overline s)]$  is the homotopy type of the pointed $k$-dimensional sphere.\hfill $\square$
\end{theorem}
\indent An isolated invariant set $K$ with respect to the semiflow $\Phi$ is called 
\emph{irreducible} if there are no isolated invariant sets $K_1$ and $K_2$ 
such that $K=K_1 \cup K_2$, $K_1\cap K_2=\emptyset$ and
\[
h(\Phi, K_1)\neq \overline{0} \quad \text{and} \quad h(\Phi, K_2)\neq \overline{0}.
\]
It is known that $K$ is irreducible if $K$ is connected, 
or $h(\Phi, K) = \overline{0}$, 
or $h(\Phi, K) = \Sigma^k$ for some integer $k \geq 0$.

\begin{theorem}[See {\cite[Th.~1.11.6]{rybakowski}}]\label{rybakowski-irreducible}
If $K_0\subset K\subset X$ are isolated invariant sets with respect to the semiflow $\Phi$ such that $K$ is irreducible, and
\[
\overline{0} \neq h(\Phi, K_0)\neq h(\Phi, K)\neq \overline{0},
\]
then there exists a full solution $u\colon \mathbb{R}\to K$ such that 
$u(\mathbb{R})\not\subset K_0$ and either $\alpha(u)\subset K_0$ or 
$\omega(u)\subset K_0$.
\end{theorem}

\end{appendices}

\vspace{10mm}

\noindent {\bf Declarations}\\

\noindent {\bf Conflict of Interest:} The authors declare that they have no conflict of interest.\\

\noindent {\bf Data Availability:} Data sharing is not applicable to this article as no datasets were generated or analysed during the current study.\\

\end{document}